\documentclass{amsart}

\usepackage{graphicx}
\usepackage{amsmath}
\usepackage[bbgreekl]{mathbbol}
\usepackage{amsfonts}
\usepackage{bbm}
\usepackage{amssymb}
\usepackage{mathtools}
\usepackage{mathrsfs}
\usepackage{verbatim}
\usepackage{color}
\usepackage{extpfeil}
\usepackage{pifont}
\usepackage{bbding}

\usepackage{tikz}
\usepgfmodule{shapes}
\usetikzlibrary{calc,positioning, matrix}

\newtheorem{thm}{Theorem}[section]
\newtheorem{cor}[thm]{Corollary}
\newtheorem{lem}[thm]{Lemma}
\newtheorem{prop}[thm]{Proposition}
\newtheorem{claim}[thm]{Claim}
\newtheorem{obs}[thm]{Observation}
\theoremstyle{definition}
\newtheorem{defn}[thm]{Definition}
\newtheorem{notation}[thm]{Notation}

\theoremstyle{remark}
\newtheorem{rem}[thm]{Remark}

\numberwithin{equation}{section}
\newcommand{\set}[1]{\left\{#1\right\}}
\newcommand{\eps}{\varepsilon}
\newcommand{\xto}{\xrightarrow}
\newcommand{\xfrom}{\xleftarrow}

\newcommand{\ob}{\mathrm{ob}}

\newcommand{\catH}{\mathscr{H}}
\newcommand{\catHfd}{\mathscr{H_{<\infty}}} 
\newcommand{\catA}{\mathscr{A}}

\newcommand{\catD}{\mathscr{D}}
\newcommand{\catV}{\mathscr{V}}
\newcommand{\Set}{Set}

\newcommand{\LM}[1]{#1 \text{-}\mathbf{Mod}} 
\newcommand{\RM}[1]{\mathbf{Mod}\text{-}#1} 
\newcommand{\ABimodH}{A\text{-}\mathbf{Mod}\text{-}A(\catH)} 

\newcommand{\bodka}{\centerdot}
\newcommand{\field}{\mathbb k}
\newcommand{\id}{\mathrm{id}}
\newcommand{\argument}{\text{\textvisiblespace}}
\newcommand{\ld}[1]{\vphantom{#1}^*#1} 
\newcommand{\rd}[1]{#1^*} 
\newcommand{\tp}{\otimes}

\newcommand{\comp}{\bullet} 
\newcommand{\daco}{\bullet} 
\newcommand{\innH}{\underline\catH}

\newcommand{\srdiecko}{\heartsuit}
\newcommand{\budzogan}{\mathop{\clubsuit}}

\newcommand{\phii}{\widehat{\phi}}

\newcommand{\Phii}{\widehat{\Phi}}
\newcommand{\Phiii}{\widehat{\widehat\Phi}}

\newcommand{\jedna}{_{(1)}}
\newcommand{\dva}{_{(2)}}
\newcommand{\jednajedna}{_{(1)(1)}}
\newcommand{\jednadva}{_{(1)(2)}}
\newcommand{\dvajedna}{_{(2)(1)}}
\newcommand{\dvadva}{_{(2)(2)}}

\newcommand{\eeps}{\mathbb{\bbespilon}}
\newcommand{\eeta}{\mathbb{\bbeta}}

\newcommand{\sipka}{\text{\ding{212}}}
\newcommand{\ams}{\mu}

\newcommand{\centre}{\mathcal{Z}}

\newcommand{\forg}{\mathrm{forg}}

\newcommand{\forme}[1]{}
\newcommand{\obr}[1]{\textcolor{blue}{obrazok #1}}

\newcommand{\cesta}{./obrazky}

\newcommand{\footnoteremember}[2]{
\footnote{#2}
\newcounter{#1}
\setcounter{#1}{\value{footnote}}
}
\newcommand{\footnoterecall}[1]{
\footnotemark[\value{#1}]
}

\begin{document}


\title{On Categories associated to a Quasi-Hopf algebra}
\author{\v Stefan Sak\'alo\v s}
\thanks{\v Stefan Sak\'alo\v s was supported by the ProDoc program ``Geometry, Algebra and Mathematical Physics'' and the grant PDFMP2\_137071 of the Swiss National Science Foundation.}
\address{Department of Mathematics, Universit\'{e} de Gen\`{e}ve, Geneva, Switzerland; Dept. of Theoretical Physics, FMFI UK, Bratislava, Slovakia}
\email{pista.sakalos@gmail.com}

%
%

\begin{abstract}
A quasi-Hopf algebra $H$ can be seen as a commutative algebra $A$ in the centre $\centre (\LM H)$ of $\LM H$. We show that the category of $A$-modules in $\centre (\LM H)$ is equivalent (as a monoidal category) to $\LM H$. This can be regarded as a generalization of the structure theorem of Hopf bimodules of a Hopf algebra to the quasi-Hopf setting.
\end{abstract}

\maketitle

\section{Introduction}
Let $H$ be a quasi-Hopf algebra. Denote $\catH := \LM H$ and $\catD:= \centre(\LM H) $.

$H$ equipped with its own adjoint action becomes an object $A\in \catH$.
If $H$ is a Hopf algebra then the left coaction by comultiplication turns $A$ into a Yetter-Drinfeld module and the original multiplication makes $A$ into a braided-commutative algebra in $\catD$. 
The construction of $A$ was generalized in \cite{Bulacu1}, \cite{Bulacu2} and \cite{Bulacu3} to work also if $H$ is a quasi-Hopf algebra. We give a category-theoretical description of $A$ in Section \ref{section_A}.

Denote by $\catA$ the monoidal category of right $A$-modules with $\tp_A$ for the monoidal structure. Our main result is:
\begin{thm}\label{thm_aaa}
$\catA$ is equivalent to $\catH$ as a monoidal category.
\end{thm}
In Section \ref{section_ordinary_Hopf} we show that if $H$ is a Hopf algebra then this statement is equivalent to the structure theorem for Hopf bimodules (see e.g. \cite{HN}). Thus one can regard Theorem \ref{thm_aaa} as a generalization of this classical result. (There is another generalization in \cite{HN}.)

Most of the article is taken up by the proof of the above theorem. In Sections \ref{section_srdiecko} and \ref{section_budzogan} we define functors 
$$\srdiecko: \catH \to \catA; \quad\quad \budzogan: \catA \to \catH$$ and show that $\budzogan \circ \srdiecko \cong \id_\catH$. We finish the proof in Section \ref{section_XiDzeta}.

The motivation for the Theorem \ref{thm_aaa} comes from the effort to understand the dequantization procedure of Etingoff and Kazhdan \cite{EK2} also for quasi-Lie bialgebras (in the spirit of \cite{SS}). 

\section*{Acknowledgement}
Pavol \v Severa introduced me to the problem and proposed the main theorem of this article. He also supported me all the time by advice and friendship.


\section{Drinfeld formulas} \label{sect_Drinfeld_formulas}
We recall some definitions from \cite{Drinfeld1}.
\begin{defn}
A quasi-bialgebra is an associative algebra $H$ together with algebra morphisms $\Delta:H \to H \otimes H$; $\eps: H\to \field$ and an invertible element $\Phi\in H\otimes H \otimes H$ satisfying 
\begin{align}
\label{B1} &(\id\otimes\Delta)\circ\Delta(a)=\Phi\cdot\big[(\Delta\otimes \id)\circ\Delta(a)\big]\cdot \Phi^{-1}\\ 
\label{B2} &(\id\otimes \id\otimes\Delta)(\Phi)\cdot(\Delta\otimes \id\otimes \id)(\Phi)=(1\otimes     \Phi)\cdot(\id\otimes\Delta\otimes \id)(\Phi)\cdot (\Phi\otimes 1)\\ 
\label{B3} &(\eps\otimes \id)\circ \Delta=\id\;;\quad\quad (\id\otimes\eps)\circ\Delta=\id\\
\label{B4} &(\id\otimes\eps\otimes \id)\Phi=1\otimes 1 
\end{align}
A quasi-Hopf algebra has in addition to the above structure an algebra antiautomorphism\forme{\footnote{Thus, unlike in the definition of a Hopf algebra, a quasi-Hopf algebra antipode $S$ is required to be invertible and the fact that $Sa\cdot Sb=S(ba)$ is a part of the definition (I don't actually know if it can't be derived from the rest.).}} $S:H\to H$ and two elements $\alpha,\beta\in H$ satisfying
\begin{align}
&\label{H1} \big(S(a_{(1)})\big)\cdot\alpha\cdot a_{(2)}=\eps(a)\cdot\alpha\\ 
&\label{H2} a_{(1)}\cdot\beta\cdot \big(S(a_{(2)})\big)=\eps(a)\cdot\beta\\
&\label{H3}\Phi^1\cdot\beta\cdot S(\Phi^2)\cdot\alpha\cdot\Phi^3=1\\
&\label{H4} S(\phi^1)\cdot\alpha\cdot\phi^2\cdot\beta\cdot S(\phi^3)=1
\end{align} 
Here we denoted $\Phi=:\Phi^1\otimes\Phi^2\otimes \Phi^3$ and $\Phi^{-1}=:\phi^1\otimes\phi^2\otimes \phi^3$ and we don't write sums $\sum$.
\end{defn}
\begin{prop} \label{prop_strict_category}
For a quasi-Hopf algebra $H$ the category of finite dimensional left $H$-modules is rigid. 

The left dual $\ld{M}$ is the usual dual vector space of $M$ equipped with the $H$-action $a\rhd f:=f(Sa\rhd\argument)$ and the evaluation and coevaluation morphisms are given by the formulas
\begin{align}
&\ld{ev}:\ld{M}\otimes M\to\field:f\otimes m\mapsto f(\alpha\rhd m)\\
&\ld{coev}:\field\to M\otimes\ld{M}:1\mapsto \sum_i(\beta\rhd e_i)\otimes e^i
\end{align}
where $\set{e_i}$ denote a vector space basis of $M$ and $\set{e^i}$ the dual basis.

The right dual $\rd{M}$ is the usual dual vector space of $M$ equipped with the $H$-action $a\rhd f:=f(S^{-1}a\rhd\argument)$ and the evaluation and coevaluation morphisms are given by the formulas
\begin{align}
&\rd{ev}:M\otimes \rd{M}\to\field:m\otimes f\mapsto f(S^{-1}\alpha\rhd m)\\
&\rd{coev}:\field\to \rd{M}\otimes M:1\mapsto \sum_i{e^i\otimes (S^{-1}\beta\rhd e_i)}
\end{align}
\end{prop}

We show in Appendix \ref{section_H_is_closed} that the category $\catH$ of all $H$-modules is closed.

\forme{
\section{The algebra $A$}
In the case when $H$ is a (non-quasi) Hopf algebra it is known that $H$ itself can be regarded as a commutative algebra $A$ in $\catD$. That is 
\begin{itemize}
\item $A=H$ as a vector space, 
\item the action of $H$ on $A$ is by left adjunction: $h\rhd a:= h_{(1)}\cdot a\cdot S(h_{(2)})$,
\item the Drinfeld-Yetter coaction is by left comultiplication
\item and finally, the product is the original product on $H$. 
\end{itemize} 
This construction has been fully generalized for $H$ quasi-Hopf in\footnote{In the first paper $A$ has been found as an algebra in $\catH$, in the second it was shown how $A$ is an object in $\catD$ and in the third the commutativity of $A$ in $\catD$ was proven.} \cite{Bulacu1}, \cite{Bulacu2} and \cite{Bulacu3}. In this section we give a category-theoretic explanation of these constructions (at least in the case when $H$ is finite dimensional and consequently $\catH$ is rigid).
\subsection{$A$ as an algebra in $\catH$} If we take whatever object $M$ of the rigid monoidal category $\catH$, then $M\otimes \ld{M}$ is an algebra in $\catH$ using the obvious multiplication:
$$\input{\cesta/obr0}\quad \comp:(M\otimes\ld{M})\otimes(M\otimes\ld{M})\to M\otimes\ld{M}$$
In particular $C\otimes \ld{C}$ is an algebra in $\catH$. As a vector space, $C\otimes\ld{C}\cong Lin(H,H)$ and we can regard $H$ as a vector subspace $A\subset Lin(H,H)$ if we assign to any $h\in H$ the left multiplication $l_h:H\to H:x\mapsto h\cdot x$. It turns out that
\begin{claim}
$A$ is an $H$-submodule of $C\otimes \ld{C}$ and is closed w.r.t. the multiplication $\comp$ on $C\otimes \ld{C}$.
\end{claim}
\begin{proof}
We do the proof by direct calculation in order to get the explicit formulae from \cite{Bulacu2}. The $H$-module action on $\ld{M}$ is given by $h \rhd\beta=\beta(Sh\rhd\argument)$ and thus the $H$-action on $M\otimes\ld{M}=Lin(M,M)$ is 
$$(h\rhd f)(\argument)=h_{(1)}\rhd f(S h_{(2)}\rhd\argument)$$
In particular the action on $l_a\in A$ is\footnote{Let's note that if we identify $a\in H$ with the corresponding $l_a\in A$ then the formula for $H$-action on $A$ becomes the same as the one we had for $H$ non-quasi:
$$h\rhd a=h_{(1)}\cdot a \cdot Sh_{(2)}$$}
$$h\rhd l_a=l_{(h_{(1)}\cdot a \cdot S h_{(2)})}$$
so we see that $A$ is really a submodule.

The formula for $f\comp g$; $f,g\in M\otimes\ld{M}\cong Lin(M,M)$ is given by the following composition:

\begin{align*}
f\otimes g 
& \xmapsto{(\daco\daco)(\daco\daco)\to \daco\big(\daco(\daco\daco)\big)} \Phi^1\rhd f(S\Phi^2\rhd\argument)\otimes\Phi^3_{(1)}\rhd g(S\Phi^3_{(2)}\rhd\argument)\\
& \xmapsto{\daco\big(\daco(\daco\daco)\big)\to\daco\big((\daco\daco)\daco\big)} \Phi^1\rhd f(S\Phi^2\cdot S\phi^1\rhd\argument)\otimes\phi^2\cdot\Phi^3_{(1)}\rhd g(S\Phi^3_{(2)}\cdot S\phi^3\rhd\argument)\\
& \xmapsto{\id_M\otimes(ev\otimes\id_{\ld{M}})} \Phi^1\rhd f(S\Phi^2\cdot S\phi^1\cdot\alpha\cdot\phi^2\cdot\Phi^3_{(1)}\rhd g(S\Phi^3_{(2)}\cdot S\phi^3\rhd\argument))
\end{align*}
If we take $f=l_a,\,g=l_b\in A$ then we see that $l_a\comp l_b=l_{a\comp b}\in A$ where we define
$$a\comp b=\Phi^1\cdot a\cdot S\Phi^2\cdot S\phi^1\cdot\alpha\cdot\phi^2\cdot\Phi^3_{(1)}\cdot b\cdot S\Phi^3_{(2)}\cdot S\phi^3$$ 
\end{proof}
}

\section{Some notations in closed monoidal categories}
Let $\catH$  be a closed monoidal category with an inner hom functor $\innH$. If we take an object $P\in ob(\catH)$ we get a pair of adjoint functors $\argument\otimes P$ and $ \innH(P,\argument)$. 
We denote the unit and counit of this adjunction by $\eeta_{M,P}$ and $\eeps_{M,P}$. They can be regarded as extranatural transformations:\footnoteremember{braids}{The significance of the two braid diagrams is of course completely different. The second is pertinent only in the rigid monoidal category and we actually should not draw it here, but it gives the idea what's going on. The lines in the first braid diagram depict just the type of the extranatural transformation and don't carry any information about the actual maps.}
$$\input{\cesta/obr1}$$
and dualy $\eeta$.

We see in Apendix \ref{section_H_is_closed} that for $\catH = \LM H$, the formulas for $\eeta$ and $\eeps$ are
\begin{gather}
\label{def_eeta}
	\eeta_{M,P}(m)(\argument)=(\phi^1\rhd m)\otimes(\phi^2\cdot\beta\cdot S\phi^3\rhd\argument) 
\\
\label{def_eeps}
	\eeps_{M,P}(f\otimes p)=\Phi^1\rhd f(S\Phi^2\cdot\alpha\cdot\Phi^3\rhd p)\;.
\end{gather}

\section{Algebra $A$}\label{section_A}
Let $\catH$ be a closed monoidal category with inner hom functor $\innH$. We denote
$$A:=\int_{X\in\catH}\innH(X,X)\,.$$
We assume in the rest of the article that the above end exists. We will see in Section \ref{section_nonsense} that it indeed does exist for $\catH = \LM H$.
\begin{lem}\noindent 
\begin{enumerate}\label{lemA}
\item $A$ is an augmented associative algebra in $\catH$.
\item Any object $M\in\catH$ is in a canonical way a left module over $A$. We denote the canonical action by $\sipka$.
\item \label{lemA3} With respect to the canonical action, every morphism in $\catH$ is $A$-linear.
\end{enumerate}
\end{lem}
\begin{proof}
The composition $\comp:\innH(Y,Z)\otimes\innH(X,Y)\to\innH(X,Z)$ is an extranatural transformation with the braid diagram:\footnoterecall{braids}
$$\input{\cesta/obr2}$$
Now $A$ as an end comes equipped with an extranatural $\alpha_X:A\to\innH(X,X)$:
$$\begin{tikzpicture}[baseline=0.0cm]
\node at (-2.5,-1.2) {$A$};
\node at (-2.5,0.0) {$\innH(X,X)$};
\draw [->] (-2.5,-1.0) ..controls +(0.0,0.33) and +(0.0,-0.33) .. (-2.5,-0.2);
\node at (-3.0,-0.6) {$\alpha_X$};
\node at (0.0,0.0) {$\catH$};
\node at (0.7,0.0) {$\catH^{op}$};
\draw  (0.0,-0.2) ..controls +(0.0,-0.67) and +(0.0,-0.67) .. (0.7,-0.2);
\end{tikzpicture}
$$
Tensoring it with itself, we get an extranatural $\alpha_X\otimes\alpha_Y:A\otimes A\to \innH(X,X)\otimes\innH(Y,Y)$:
$$\input{\cesta/obr4}$$
We compose the last with $\comp$ to get an extranatural
$$\input{\cesta/obr5}$$
This is actually an extranatural transformation to $\innH(X,X)$ of type \begin{tikzpicture}[baseline=-0.5cm]
\node at (0.0,0.0) {{\scriptsize$\catH$}};
\node at (0.4,0.0) {{\scriptsize$\quad\catH^{op}$}};
\draw  (0.0,-0.2) ..controls +(0.0,-0.33) and +(0.0,-0.33) .. (0.4,-0.2);
\end{tikzpicture}
 and by the universal property of an end there exists a unique $A\otimes A\xto{\comp} A$ such that
$$\input{\cesta/obr7}$$
commutes for all $X\in ob(\catH)$.
So $\alpha_X:A\to\catH(X,X)$ is by definition an algebra homomorphism what gives us the action $\sipka$ of $A$ on $X$.\footnote{Recall that the action of $\catH(X,X)$ on $X$ is given by $\eeps_{X,X}:\innH(X,X)\otimes X\to X$.} Looking at this definition of $\sipka$ we see that $\sipka_X:A\otimes X\to X$ is a composition of two extranatural transformations:
$$\input{\cesta/obr8}$$
and thus a natural transformation. But the fact that it is a natural transformation actually means that all $\catH$-morphisms are $A$-linear.
\end{proof}
\begin{notation}
$A$ has an augmentation $\eps_A: A \to I$ defined to be equal to $\sipka_I: A\tp I \to I$. On pictures, we denote it by \begin{tikzpicture}[baseline=0.0cm]
\filldraw [line width=1pt] (0.0,0.0) circle (0.08cm);
\node at (0.0,-0.6) {{\scriptsize $A$}};
\draw [line width=1pt] (0.0,-0.4) ..controls +(0.0,0.17) and +(0.0,-0.17) .. (0.0,-0.08);
\end{tikzpicture}
 .
\end{notation}
\begin{lem}
If $\catH$ is rigid or if it is the category of representations of a quasi-Hopf algebra then the map $\catH(N,M\otimes A)\to Nat(N\otimes\argument\,, M\otimes\argument)$ given by
$$\input{\cesta/obr10}\mapsto\input{\cesta/obr11}$$
is an isomorphism.
\end{lem}
\begin{proof}
We will prove a more general Lemma \ref{lem_vlocka} later.
\end{proof}
In the rest of the article we assume that the conclusion of the lemma holds and thus we have the above bijection between $Nat(N\otimes\argument\,, M\otimes\argument)$ and $\catH(N,M\otimes A)$. Using this assumption we can for any $M\in \catH$ define a ``braiding between A and M'' i.e. an $\catH$-morphism 
$$\beta_{A,M}: A\otimes M\to M\otimes A$$
by requiring for any $T\in \catH$
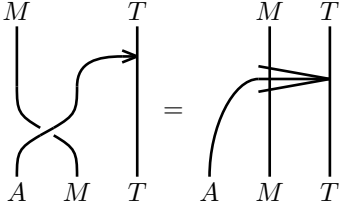
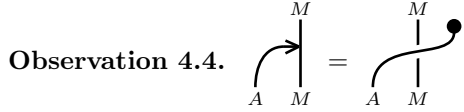
$$\input{\cesta/obr12}=\input{\cesta/obr13}$$
\begin{obs}\label{obs_sipka_a_epsA} $ \input{\cesta1/7daco7} = \input{\cesta1/7daco6}$
\end{obs}
\begin{lem}  \label{lem_A_is_in_Double}
Using the above maps $A$ becomes a commutative algebra in the double $\catD$ of our $\catH$.
\end{lem}
\begin{proof}
For $A$ to be an object of $\catD$ we need $\beta_{A,M}$ to be natural in $M$ and to satisfy the hexagon.
Naturality in $M$:\footnote{The second equality is from $A$-linearity of all $\catH$-morphisms. The rest is the definition of $\beta$.}
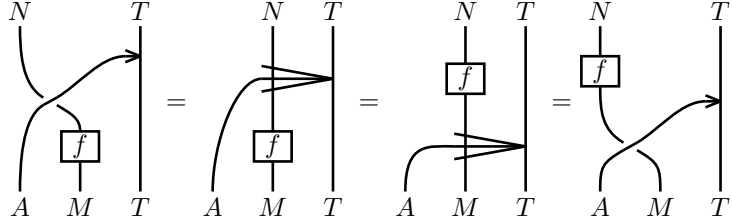
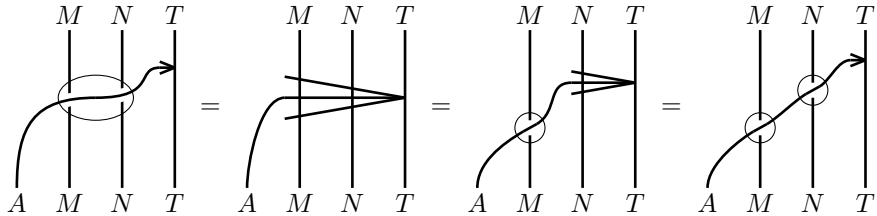
$$\input{\cesta/obr14}=\input{\cesta/obr15}=\input{\cesta/obr16}=\input{\cesta/obr17}$$
The hexagon:\footnote{Here we just use three times the definition of $\beta$.}
$$\input{\cesta/obr18}=\input{\cesta/obr19}=\input{\cesta/obr20}=\input{\cesta/obr21}$$

For $A$ to be an algebra in $\catD$ we need the compatibility between the braiding and the multiplication.
Doing the multiplication first and the the braiding corresponds to\footnote{The equality is just the definition of $\beta$.}
$$\input{\cesta/obr22}=\input{\cesta/obr23}$$
Doing the braiding first and then multiplying gives:\footnote{The first equality comes from the fact that $\sipka$ is an action and the second is the definition of $\beta$ applied twice.}
$$\input{\cesta/obr24}=\input{\cesta/obr25}=\input{\cesta/obr26}$$
So we see that the two results are the same because $\sipka$ is an action. 

The last thing is the commutativity of $A$:\footnote{The first and the last equalities hold because $\sipka$ is an action, the second is the definition of $\beta$ and the third is from $A$-linearity of any $\catH$-morphism.}
$$\input{\cesta/obr27}=\input{\cesta/obr28}=\input{\cesta/obr29}=\input{\cesta/obr30}=\input{\cesta/obr31}$$
\end{proof}

\forme{
\section{category $\catA:= \RM A (\catD)$}
\begin{enumerate}
\item prove that $\catD$ satisfies what it should if $\catH$ does.
\item the forgetful functor $\catD \to \catH$ preserves limits and thus it does not matter whether we do $\otimes_A$ in $\ABimodH$ or in $\catA$.
\end{enumerate}
}


\section{Abstract Nonsense}\label{section_nonsense}
\begin{defn}
Let $\catH$, $\catA$ be any categories and $F:\catH^{op}\times\catH\to\catA$ be a functor. We will say that an object $C\in\catH$ generates $F$ if for any $D\in ob(\catA)$, $X,Y\in ob(\catH)$ and two different maps $\delta_1,\delta_2:D\to F(X,Y)$ there exist $f\in\catH(C,X)$ s.t. the two compositions $D\xto{\delta_i}F(X,Y)\xto{F(f,\id_Y)} F(C,Y)$ are different. 
\end{defn}
\begin{rem}
If $F=Hom_\catH:\catH^{op}\otimes\catH\to \Set$ then ``$C$ generates $F$'' just means that $C$ is a generator of the category $\catH$.
\end{rem}
The only reason for writing the above definition is the following
\begin{prop}
Let $\catH$, $\catA$ be categories and $F:\catH^{op}\times\catH\to\catA$ a bifunctor generated by $C$. Let $D\in \catA$ be an object and $\delta:D\to F$ an extra-natural transformation. If $(D,\delta)$ is an end of $F|_C$ then it is an end of $F$.
\end{prop}
\begin{proof}
Let $D'\in \catA$ and $\delta':D'\to F$ be another extranatural transformation. We want to show that $\delta'$ factorizes through $\delta$. Well, since $D\xto{\delta_C}F(C,C)$ is the end of $F|_C$, we have a unique $h\in\catA(D',D)$ s.t.
$$\begin{tikzpicture}
	\node(vzdx) at (2, 0){};
	\node(vzdy) at (0, 1.5){};
	
	\node(a) {$D'$};
	\node(b) at ($(a)+(vzdy)$) {$D$};
	\node(c) at ($(b)+(vzdx)$) {$F(C,C)$};
	
	\draw[->] (a)-- node[left] {$\exists ! h$} (b);
	\draw[->] (a)-- node[below] {$\delta'_C$} (c);
	\draw[->] (b)-- node[above] {$\delta_C$} (c);	
\end{tikzpicture}
$$
We just have to show that $\forall X\in\catH$ we also have commutativity of
$$\begin{tikzpicture}
	\node(vzdx) at (2, 0){};
	\node(vzdy) at (0, 1.5){};
	
	\node(a) {$D'$};
	\node(b) at ($(a)+(vzdy)$) {$D$};
	\node(c) at ($(b)+(vzdx)$) {$F(X,X)$};
	
	\draw[->] (a)-- node[left] {$h$} (b);
	\draw[->] (a)-- node[below] {$\delta'_X$} (c);
	\draw[->] (b)-- node[above] {$\delta_X$} (c);	
\end{tikzpicture}
$$
Now for any $C\xto{f} X$ in $\catH$ we have a diagram
$$\input{\cesta/obr34}$$
where everything commutes except possibly our desired triangle. From that we get that our triangle commutes after it is composed with any $F(f,1)$:
$$\input{\cesta/obr35}$$
Now we just apply the fact that $C$ generates $F$.
\end{proof}
We have the following corollary:
\begin{lem}\label{lem_explicit_srdiecko}
Let again $H$ be an quasi-Hopf algebra, $\catH:=\text{$H$-Mod}$ and let $C\in ob(\catH)$ be the left regular module of $H$. Then $C$ generates $\innH$ and thus
\begin{align*}
\int_{X\in\catH}\innH(X,X)&=\bigcap_{f\in\catH(C,C)}Ker\big(\innH(C,C)\xto{f^\ast-f_\ast}\innH(C,C)\big)=\\
&=\big\{l_a\in Lin(H,H)\,|\;a\in H \big\}
\end{align*}
where $l_a$ denotes the left multiplication $l_a:h\mapsto a\cdot h:H\to H$.

More generally for any $M,N\in ob(\catH)$ our $C$ generates the bifunctor 
$$X,Y\mapsto \innH(X,M\otimes Y\otimes N)$$ and thus
\begin{align*}
&\int_{X\in\catH}\innH(X,M\otimes X\otimes N) =\\
&\quad\quad\quad=\bigcap_{f\in\catH(C,C)}Ker\big(\innH(C,M\otimes C\otimes N)\xto{f^\ast-f_\ast}\innH(C,M\otimes C\otimes N)\big)=\\
&\quad\quad\quad=\big\{l_{\sum m_i\otimes a_i\otimes n_i}\in Lin(H,M\otimes H\otimes N)\,|\;\sum m_i\otimes a_i\otimes n_i\in M\otimes H\otimes N \big\}
\end{align*}
where $l_{\sum m_i\otimes a_i\otimes n_i}:H\to M\otimes H\otimes N: h\mapsto \sum m_i\otimes (a_i\cdot h)\otimes n_i$
\end{lem}


\section{Functor $\srdiecko:\catH \to \catA$}\label{section_srdiecko}
Generalizing the construction of $A$ define the functor $\srdiecko:\catH\to\catH$:
\begin{equation}\label{equa_srdiecko}
\srdiecko M=\int_{X\in \catH}\innH(X,X\otimes M)\;.
\end{equation}
We will further assume that the above limit exists. We know from Lemma \ref{lem_explicit_srdiecko} that this is the case for $\catH = \LM H$.
\begin{lem}\noindent
\begin{enumerate}
\item There is a natural "associative product" (in other words, a (lax) monoidal structure on $\srdiecko : \catH \to \catH$): $$\comp: \srdiecko M\otimes\srdiecko N\to\srdiecko (M\otimes N)\,.$$
\item There is an "action" $$\Diamond:\srdiecko M\otimes X\to X\otimes M$$ that is natural\footnote{Naturality of $\Diamond$ in $X$ corresponds to $A$-linearity of all $\catH$-morphisms in the part \ref{lemA3} of Lemma \ref{lemA}.} in both arguments. It is an action in the sense that
$$\input{\cesta/obr36}=\input{\cesta/obr37}$$
\end{enumerate}
\end{lem}
\begin{proof}
The extranatural transformation 
$$\comp:\innH(Y,Z\otimes M)\otimes\innH(X,Y\otimes N)\to\innH(X,Z\otimes M\otimes N)$$
gives rise to $\comp: \srdiecko M\otimes\srdiecko N\to\srdiecko (M\otimes N)$ in the same way as in the definition of $\comp: A\otimes A\to A$. Similarly we get $\Diamond:\srdiecko M\otimes X\to X\otimes M$ from the extranatural
$$\eeps_{M\otimes Y,X}:\innH(X,M\otimes Y)\otimes X\to M\otimes Y$$
in the same way as we defined the action $\sipka$.
\end{proof}
\begin{notation}
Generalising the definition of $\eps_A$ we can define a natural projection $\pi_M: \srdiecko M \to M$ to be equal to $\Diamond_{M,I}: \srdiecko M \tp I \to M \tp I$.
\end{notation}
\begin{rem}\label{rem_jalaa}
We see that $A=\srdiecko I$ and that $\comp$ gives us a left and a right $A$-action on any $\srdiecko M$. \footnote{We will see in (\ref{equa_commutativity_of_srdiecko}) how these two actions are related. Let us also remind the reader that there is yet a third action $\sipka:A\otimes \srdiecko M\to \srdiecko M$.} From associativity of $\comp$ we get that these two actions commute and thus $\srdiecko M$ is an $A$-$A$-bimodule. Another use of the associativity of $\comp$ gives us
$$\input{\cesta/obr38}=\input{\cesta/obr39}$$
what means that $\comp:\srdiecko M\otimes\srdiecko N\to \srdiecko(M\otimes N)$ is $A$-bilinear and thus induces an $A$-$A$-morphism\ $\srdiecko M\otimes_A\srdiecko N\to \srdiecko N$. This natural transformation turns $\srdiecko$\ into a lax\footnote{The rest of the article will actually imply that it is a strong monoidal functor at least for $\catH = \LM H$.}  monoidal functor $\srdiecko:\catH\to A\text{-Mod-}A$ where the product in $A\text{-Mod-}A$ is the tensor product over $A$. 
\end{rem}

Again we need a technical
\begin{lem}\label{lem_vlocka}
If $\catH$ is rigid or if it is the category of representations of a quasi-Hopf algebra then the map $\catH(X,Y\otimes \srdiecko M)\to Nat(X\otimes\argument\,, Y\otimes\argument\otimes M)$ given by
$$\input{\cesta/obr40}\mapsto\input{\cesta/obr41}$$
is an isomorphism.
\end{lem}
\begin{rem}
Note that if there is no $Y$ (i.e. $Y= I$) then the lemma holds in any  category. Really, $$\catH( X \tp T, T \tp M ) = \catH\big( X , \innH( T, T\tp M) \big) $$ so the natural (in $T$) transformations $X \tp T \to T \tp M $ correspond to the dinatural (in $T$) transformations $X \to \innH (T,T\tp M) $ and thus to maps $X \to \srdiecko M$.
\end{rem}
\begin{proof}[Proof of Lemma \ref{lem_vlocka} for a rigid $\catH$]
We have  an  isomorphism 
$$ \catH( X \tp T , Y \tp T \tp M ) \cong \catH \big( \ld Y \tp X , ( T \tp M ) \tp \ld T \big) = \catH \Big( \ld Y \tp X , \innH ( T, T \tp M ) \Big) $$
$$ \input{\cesta1/5daco1} \mapsto \input{\cesta1/5daco2}$$
Thus the natural transformations $X \tp T \to Y\tp T \tp M $ correspond to dinatural transformations $ \ld Y \tp X \to \innH( T, T \tp M)$ that in turn correspond to maps $\ld Y \tp X \to \srdiecko M$ which  is the same as maps $ X \to Y \tp \srdiecko M $.
\end{proof}
\begin{claim}
The conclusion of Lemma \ref{lem_vlocka} holds provided the following map is iso:  
\begin{equation}
\int_{T\in\catH}\innH(T,Y\otimes T\otimes M)\xfrom{\cong}Y\otimes \srdiecko M.
\end{equation}
\end{claim}
\begin{proof}
\begin{align*}
Nat(X\otimes\argument\,, Y\otimes \argument\otimes M)&=\int_{T\in\catH}\catH(X\otimes T,Y\otimes T \otimes M)=\\
&=\int_{T\in\catH}\catH\big(X,\innH(T,Y\otimes T\otimes M)\big)=\\
&=\catH\big(X,\int_{T\in\catH}\innH(T,Y\otimes T\otimes M)\big)\xfrom{\cong}\\
&\xfrom{\cong}\catH(X, Y\otimes \srdiecko M)
\end{align*}
\end{proof}
\begin{proof}[Proof of Lemma \ref{lem_vlocka} for $\catH = \LM H$]
Just use the above claim and Corollary \ref{cor_skjfss}.
\end{proof}

In the rest of the article we assume that we have the above bijection between $Nat(X\otimes\argument\,, Y\otimes\argument\otimes M)$ and $\catH(X,Y\otimes \srdiecko M)$. Using this assumption we can for any $X\in ob(\catH)$ define a ``braiding between $\srdiecko M$ and $X$'' i.e. an $\catH$-morphism 
$$\beta_{\srdiecko M,X}: \srdiecko M\otimes X\to X\otimes \srdiecko M$$
by requiring
$$\input{\cesta/obr42}=\input{\cesta/obr43}$$
\begin{lem}\label{lem_ajhfka}
Using the above maps $\srdiecko M$ becomes an object in the centre $\catD$ of our $\catH$ and we can regard \ $\srdiecko$ as a functor \ $\srdiecko: \catH\to\catD$. The ``product'' $\comp:\srdiecko M \otimes \srdiecko N\to \srdiecko(M\otimes N)$ is a $\catD$-morphism. Moreover we have
\begin{equation} \label{equa_commutativity_of_srdiecko}
\input{\cesta/obr44}=\input{\cesta/obr45}
\end{equation}
\end{lem}
\begin{rem}
Note however that in general
$$\input{\cesta/obr46} \neq \input{\cesta/obr45}$$
although the equality takes place when $M=I$ because then $\srdiecko M=\srdiecko I=A$.
\end{rem}
\begin{proof}
The proof is the same as that of lemma \ref{lem_A_is_in_Double}. For example the proof of the commutativity of $A$ translates into the following proof of the equation (\ref{equa_commutativity_of_srdiecko}):
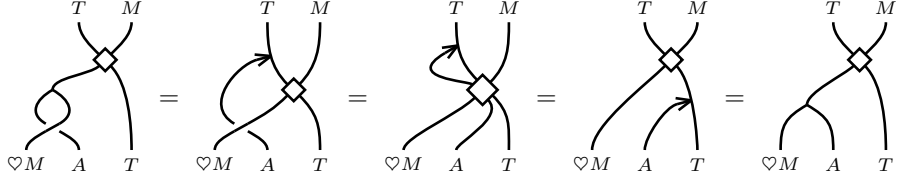
$$\input{\cesta/obr47}=\input{\cesta/obr48}=\input{\cesta/obr49}=\input{\cesta/obr50}=\input{\cesta/obr51}$$
\end{proof}

We denote by $\catA$ the category of right $A$-modules in $\catD$. For $X\in \ob(\catA)$, the right $A$-action will be denoted
$$\input{\cesta1/1daco2A}$$
We can turn a right $A$-module in $\catD$ into a left $A$-module by
$$\input{\cesta1/1daco2opacne}  := \input{\cesta1/5daco3} $$
Thus we can equip $\catA$ with $\tp_A$. Remark \ref{rem_jalaa} together with Lemma \ref{lem_ajhfka} now imply that we can see $\srdiecko$ as a monoidal functor $\srdiecko: \catH \to \catA$.

\begin{lem}
If $M\in ob(\catD)$ then $\srdiecko M$ is canonically isomorphic in $\catA$ to the free $A$-module $M\otimes A$. This gives a monoidal\footnote{One must specify the monoidal structure on $\catD \xto{\argument \tp A} \catA$. We take the isomorphism $(M\tp A) \tp_A (N \tp A) \to (M \tp N) \tp A$ induced by the map $(M\tp A) \tp (N \tp A) \to (M \tp N) \tp A$ in which the first $A$ goes behind $N$ to join the second $A$.} natural isomorphism between $\catD \xto\forg \catH \xto\srdiecko \catA$ and $\catD \xto{\argument \tp A} \catA$.
\end{lem}
\begin{proof}
One has $\srdiecko M=\int_{X\in \catH}\innH(X,X\otimes M)=\int_{X\in \catH}\innH(X,M\otimes X)\cong M\otimes A$. However to prove the fact that the isomorphism is in fact in $\catD$ we need a more explicit description. Define the natural maps $s_M:\srdiecko M \to M \tp A$ and $t_M : M \tp A \to \srdiecko M$ by
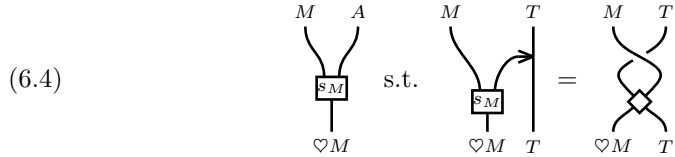
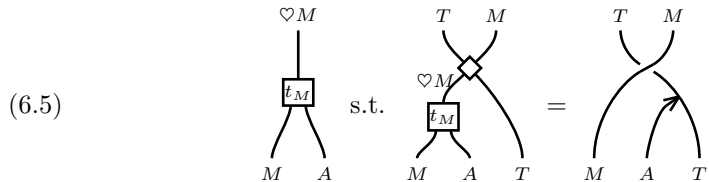
\begin{equation}\label{equa_sM}
\input{\cesta/obr52}\text{ s.t. }\input{\cesta/obr53}=\input{\cesta/obr54}
\end{equation}
\begin{equation}\label{equa_tM}
\input{\cesta/obr55}\text{ s.t. }\input{\cesta/obr56}=\input{\cesta/obr57}
\end{equation}
Using our pictures, one can directly check that these maps are mutually inverse and that they are compatible with the braidings (i.e. they are in $\catD$) and right $A$-linear. For instance the proof of the compatibility with braidings looks like this:
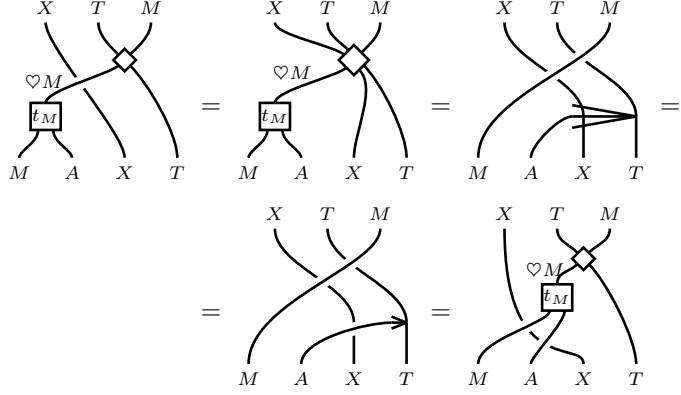
\begin{align*}
\input{\cesta/obr58}&=\input{\cesta/obr59}=\input{\cesta/obr60}= \\
& = \input{\cesta/obr61}=\input{\cesta/obr62}
\end{align*}
\end{proof}


\section{ordinary Hopf algebra case}\label{section_ordinary_Hopf}
In this section, $\catH= \LM H$ for a Hopf algebra $H$. In this case we show that the theorem \ref{thm_aaa} becomes the well known structure theorem for Hopf modules.
\begin{rem}
In the other sections, the diagrams are in $\catH$, i.e. all the morphisms used in them are in $\catH$. In this section, we allow diagrams in the underlying category of vector spaces. In particular, we can interchange two objects:
$$
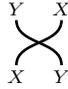
\begin{tikzpicture}[baseline=0.4cm]
\node at (0.0,0.0) {{\scriptsize $X$}};
\node at (0.6,0.0) {{\scriptsize $Y$}};
\node at (0.0,0.9) {{\scriptsize $Y$}};
\node at (0.6,0.9) {{\scriptsize $X$}};
\draw [line width=1pt] (0.0,0.15) ..controls +(0.0,0.33) and +(0.0,-0.33) .. (0.6,0.75);
\draw [line width=1pt] (0.6,0.15) ..controls +(0.0,0.33) and +(0.0,-0.33) .. (0.0,0.75);
\end{tikzpicture}
$$
\end{rem}

\subsection{Algebra $A$.} From lemma \ref{lem_explicit_srdiecko} we see that $A$ can be identified with $H$ that is an $\catH$-object via the $H$-action by adjunction. The Drinfeld-Yetter coaction is given by the left $H$-comultiplication. The multiplication in $A$ is just the multiplication in $H$. 

\subsection{Objects of $\catA$}
An object $M\in \catA$ is a vector space with left action, left coaction and right action\footnote{We get the first one because $M\in\catH=\LM H$, the second one is the Drinfeld-Yetter coaction ($M\in\catD$) and the third is the right $A$-action.}
\begin{equation}\label{tri_operacie}
 \input{\cesta1/4daco1} \quad \input{\cesta1/4daco2} \quad  \input{\cesta1/1daco2} 
\end{equation}
that satisfy\footnote{We get (\ref{C1}) because $M\in \catD$ is a Drinfeld-Yetter module, (\ref{C2}) means that the $A$-module structure is $H$-linear and (\ref{C3}) holds because the $A$-module structure is a $\catD$-morphism and thus compatible with the $H$-comodule structure. }
\begin{align*}
\end{align*}
\begin{align}
\label{C1}   \input{\cesta1/6daco2} = \input{\cesta1/6daco1} \\
\label{C2}  \input{\cesta1/6daco4} = \input{\cesta1/6daco3} \\
\label{C3}  \input{\cesta1/6daco5} = \input{\cesta1/6daco6} 
\end{align}
\subsection{Objects of $\catA$ as Hopf-bimodules}
Let's introduce a new left $H$-action
$$
\input{\cesta1/4daco3} := \input{\cesta1/6daco7}
$$
Then the operations 
\begin{equation}\label{ine_tri_operacie}
 \input{\cesta1/4daco3} \quad \input{\cesta1/4daco2} \quad  \input{\cesta1/1daco2H} 
\end{equation}
still contain the same information as did (\ref{tri_operacie}). The conditions satisfied by this new set of operations are
\begin{align}
 \input{\cesta1/6daco11} &= \input{\cesta1/6daco10} \\
 \input{\cesta1/6daco8} &= \input{\cesta1/6daco9}
\end{align}
and the old condition (\ref{C3}). But this means that $M$ together with the operations (\ref{ine_tri_operacie}) is a Hopf-bimodule.
\subsection{Functor $\srdiecko:\catH \to \catA$} For $X\in \catH$ we take $\srdiecko X:= H\tp X$ with the above operations defined by
\begin{equation}\label{Hopf_srdiecko}
\begin{array}{lll}
a) \input{\cesta1/5daco7} := \input{\cesta1/5daco6} &\quad& b) \input{\cesta1/5daco11} := \input{\cesta1/5daco10} \\
c)  \input{\cesta1/5daco5} := \input{\cesta1/5daco4} && d) \input{\cesta1/5daco9} := \input{\cesta1/5daco8}
\end{array}
\end{equation}
Thus we see that the fact that $\srdiecko:\catH \to \catA$ is an equivalence of categories can be reformulated as
\begin{thm}
The functor from $H$-modules to $H$-Hopf-bimodules given by $X\mapsto H\tp X$ and the operations b), c), d) in (\ref{Hopf_srdiecko}) is an equivalence of categories. 
\end{thm} 
And this is the well known structure theorem of Hopf-bimodules.


\section{The functor $\budzogan: \catA \to \catH$}\label{section_budzogan}
The unit object $I\in \catH$ is an $A$-module via $\sipka$.\footnote{Equivalently, $I$ is an $A$-module thanks to the augmentation $\eps_A$ on $A$ in $\catH$.}
For $M\in \catA$ define\footnote{This tensor product must be taken in $\catH$ since the $A$-module structure of $I$ is not in $\catD$.} 
$$\budzogan M := M \tp_A I.$$

\begin{prop}
$\budzogan$ is a strong monoidal functor.
\end{prop}
\begin{proof}
We should show that for $M, N \in \catA$ we have an $\catH$-isomorphism 
$$(M\tp_A I)\tp (N \tp_A I) \;\cong\; (M \tp_A N) \tp_A I .$$
And really, both sides can be characterized as a universal $\catH$-object $U$ together with an $H$-morphism $M\tp N \xto u U$ satisfying\footnote{There is a subtle point here. To show that $(M\tp_A I)\tp (N \tp_A I)$ satisfies (\ref{equa_fkjds}) one needs to use 
$$ \input{\cesta1/7daco4} = \input{\cesta1/7daco5}$$ This would not hold if $N$ crossed under $A$.}
\begin{equation}\label{equa_fkjds}
\input{\cesta1/7daco1} = \input{\cesta1/7daco2}= \input{\cesta1/7daco3}
\end{equation}
\end{proof}

\begin{lem}
For $M \in \catH$, the $\catH$-morphism $\Diamond_{M,I}: \srdiecko M \tp I \to M\tp I= M$ is inner $A$-bilinear and thus induces an $\catH$-morphism $\srdiecko M \tp_A I \to M$. This gives a monoidal natural transformation $\budzogan\circ \srdiecko\to\id_\catH$.

If $\catH$ is the category of representations of a quasi-Hopf algebra then it is an isomorphism.
\end{lem}
\begin{proof}
We want to show that $\srdiecko M \tp_A I \xto\cong M$ if $\catH = \LM H$. This is equivalent to the exactness of the upper row in the diagram (\ref{diag_slnko}). To prove it, we compare $\srdiecko M$ to the free $A$-module $M\tp A$. It is clear that (in $\catH$) we have $(M \tp A)\tp_A I \cong M$ which translates as the exactness of the bottom row of the diagram (\ref{diag_slnko}). Now we want to join the two rows by vertical $\catV$-isomorphisms into a commutative diagram
\begin{equation}\label{diag_slnko}
\begin{tikzpicture}[label distance =-0.5 mm, baseline=-1 cm]	
	\node(lx1) at (6,0){};
	\node(lx2) at (2.5,0){};
	\node(lx3) at (1.5,0){};
	\node(lx4) at (3.1,0){};
	\node(ly) at (0, -2){};
	\pgfresetboundingbox;

	\node(a00) {$\srdiecko M \tp A$};
	\node(a01) at ($(a00)+(lx1)$) {$\srdiecko M$};
	\node(a02) at ($(a01)+(lx2)$) {$M$};
	\node(a03) at ($(a02)+(lx3)$) {$0$};
	
	\node(a10) at ($(a00)+(ly)$) {$(M \tp A) \tp A$};
	\node(a11) at ($(a01)+(ly)$) {$ M \tp A$};
	\node(a12) at ($(a02)+(ly)$) {$M$};
	\node(a13) at ($(a03)+(ly)$) {$0$};

	\draw[->] (a00)-- node[above] {$\comp - \id\tp \eps_A$} (a01);
	\draw[->] (a01)-- node[above] {$\pi_M$} (a02);
	\draw[->] (a02)-- node[above] {} (a03);

	\draw[<-] (a00)-- node[left] {$\cong$} (a10);
	\draw[<-] (a01)-- node[left] {$\cong$} (a11);
	\draw[<-] (a02)-- node[left] {$=$} (a12);
	
	\draw[->] (a10)-- node[below] {$( \id_M\tp \comp) \circ \Phi - \id \tp \eps_A$} (a11);
	\draw[->] (a11)-- node[below] {$  \id \tp\eps_A$} (a12);
	\draw[->] (a12)-- node[below] {} (a13);

\end{tikzpicture}
\end{equation}
which will prove the exactness of the upper row and thus our lemma. 

We define the middle vertical arrow as 
$$ M\tp A\to \srdiecko M  :\; m\tp a \mapsto a \tp m. $$
Using the formulas for $\eps_A$ and $\pi_M$ from Lemma \ref{lem_kopec_formul} we see that the right window of (\ref{diag_slnko}) commutes.

The last thing is to define the left vertical isomorphism that would make the left window commute.  Recall from Corollary \ref{cor_product_formulas} that the $A$-module structure $\srdiecko M \tp A \xto\comp \srdiecko M$ on $\srdiecko M$ is 
$$ ( a \tp m ) \comp b \; = \; \Big( \kappa^1 \cdot a \cdot S\kappa^2 \cdot \alpha \cdot \kappa^3 \cdot b \cdot S\kappa^4 \Big) \tp \big( \kappa^5 \rhd m \big) $$
where $\kappa$ is\footnoteremember{kappalambda}{Explicitly: $\kappa = (1\tp \phi \tp 1) \cdot \Phi_{(1,5),2,(3,4)} $ and $\lambda = (1\tp 1 \tp \phi) \cdot \Phi_{2,3,(4,5)} \cdot \Phi_{1,(2,3),(4,5)} $.} an invertible element of $H^{\tp 5} $. Again using Corollary \ref{cor_product_formulas}, the $A$-module structure $(M\tp A) \tp A \xto{ (\id_M \tp \comp ) \circ \Phi} M \tp A $ on $M\tp A$ is
$$ (m\tp a) \tp b \; \mapsto \; \big(\lambda^1 \rhd m \big) \tp \Big( \lambda^2 \cdot a \cdot S\lambda^3 \cdot \alpha \cdot \lambda^4 \cdot b \cdot S\lambda^5 \Big) $$
for an invertible\footnoterecall{kappalambda} $\lambda\in H^{\tp 5}$. Thus if we define the left vertical arrow $( M \tp A ) \tp A \to  \srdiecko M \tp A$ in (\ref{diag_slnko}) by (we denote $\bar\kappa := \kappa^{-1}$)
$$ (m \tp a) \tp b \mapsto \Big[ \bar\kappa^1 \cdot \lambda^2 \cdot a \cdot S( \bar\kappa^2 \cdot \lambda^3) \Big] \tp \Big[ \bar\kappa^5 \cdot \lambda^1 \rhd m \Big] \tp \Big[ \bar\kappa^3 \cdot \lambda^4 \cdot b \cdot S( \bar\kappa^4 \cdot \lambda^5) \Big] $$
then the left square in (\ref{diag_slnko1}) will commute.
\begin{equation}\label{diag_slnko1}
\begin{tikzpicture}[label distance =-0.5 mm, baseline=-1 cm]	
	\node(lx1) at (4,0){};
	\node(lx2) at (2.5,0){};
	\node(lx3) at (1.5,0){};
	\node(lx4) at (3.1,0){};
	\node(ly) at (0, -2){};
	\pgfresetboundingbox;

	\node(a00) {$\srdiecko M \tp A$};
	\node(a01) at ($(a00)+(lx1)$) {$\srdiecko M$};
	
	\node(a10) at ($(a00)+(ly)$) {$(M \tp A) \tp A$};
	\node(a11) at ($(a01)+(ly)$) {$ M \tp A$};

	\draw[->] (a00)-- node[above] {$\comp $} (a01);

	\draw[<-] (a00)-- node[left] {$\cong$} (a10);
	\draw[<-] (a01)-- node[left] {$\cong$} (a11);
	
	\draw[->] (a10)-- node[below] {$( \id_M\tp \comp) \circ \Phi $} (a11);

\end{tikzpicture}
\quad
\begin{tikzpicture}[label distance =-0.5 mm, baseline=-1 cm]	
	\node(lx1) at (3.3,0){};
	\node(lx2) at (2.5,0){};
	\node(lx3) at (1.5,0){};
	\node(lx4) at (3.1,0){};
	\node(ly) at (0, -2){};
	\pgfresetboundingbox;

	\node(a00) {$\srdiecko M \tp A$};
	\node(a01) at ($(a00)+(lx1)$) {$\srdiecko M$};
	
	\node(a10) at ($(a00)+(ly)$) {$(M \tp A) \tp A$};
	\node(a11) at ($(a01)+(ly)$) {$ M \tp A$};

	\draw[->] (a00)-- node[above] {$ \id\tp \eps_A$} (a01);

	\draw[<-] (a00)-- node[left] {$\cong$} (a10);
	\draw[<-] (a01)-- node[left] {$\cong$} (a11);
	
	\draw[->] (a10)-- node[below] {$ \id \tp \eps_A$} (a11);

\end{tikzpicture}
\end{equation}
To prove that the right square in (\ref{diag_slnko1}) also commutes just use the formula for $\eps_A$ from Lemma \ref{lem_kopec_formul} and\footnote{To show these two equations, apply the claim \ref{claim_eps_on_Phi} to the explicit formulas\footnoterecall{kappalambda} for $\kappa$ and $\lambda$.}
\begin{gather*}
\kappa^1 \tp \kappa^2 \tp \eps\kappa^3 \tp \eps\kappa^4 \tp \kappa^5 = 1_H \tp 1_H \tp 1_H \\
\lambda^1 \tp \lambda^2 \tp \lambda^3 \tp \eps\lambda^4 \tp \eps\lambda^5  = 1_H \tp 1_H \tp 1_H.
\end{gather*}
\end{proof}


\section{Isomorphism $\srdiecko \circ \budzogan \; \cong \; \id_\catA$}\label{section_XiDzeta}
In this section, we assume $\catH = \LM H$ for a quasi-Hopf algebra $H$ and we fix $M \in \catA$. We denote by 
$$ \ams :M\tp A \to M \quad\text{and} \quad p_M: M \to \budzogan M $$
the right $A$-module structure on $M$ and the natural projection. Recall also isomorphisms
 $$ s_M:\srdiecko M \xto\cong M \tp A \quad \quad t_M : M \tp A \xto\cong \srdiecko M$$
defined by formulas (\ref{equa_sM}) and (\ref{equa_tM}).
\subsection{Definition of $\xi:\; \srdiecko\circ \budzogan \to \id_\catA$} 
\begin{claim}
$ \Xi_M :\; \srdiecko M \xto{s_M} M\tp A \xto\ams M $ is in $\catA$.
\end{claim}
\begin{proof}
Both maps are $A$-linear if the $A$-module structure on $M\tp A$ is given just by multiplication on $A$.
\end{proof} 
By Lemma \ref{lem_explicit_srdiecko}, $\srdiecko$ is right-exact and thus exactness of (sequence in $\catH$)
$$ M\tp A \xto{ \ams - (\id_M \tp \eps_A)} M  \xto{p_M} \budzogan M \to 0 $$
implies exactness of (sequence in $\catA$)
\begin{equation}\label{seq1}
 \srdiecko( M\tp A )\xto{ \srdiecko\ams - \srdiecko(\id_M \tp \eps_A)}  \srdiecko M  \xto{\srdiecko p_M} \srdiecko\budzogan M \to 0 
\end{equation}
\begin{claim}
$\Xi_M$ uniquely factorizes as $ \srdiecko M \xto{\srdiecko p_M} {\srdiecko \budzogan M } \xto{\xi_M} M$ with $\xi_M\in \catA$.
\end{claim}
\begin{proof}
By exactness of (\ref{seq1}), it is enough to verify that 
$$ \Big( \srdiecko(M\tp A) \xto{\srdiecko \ams } \srdiecko M \xto{\Xi_M} M \Big)\; \xlongequal{}\; \Big(\srdiecko(M\tp A) \xto{\srdiecko( \id_M\tp \eps_A) } \srdiecko M \xto{\Xi_M} M \Big). $$
This equality becomes obvious once we insert the two sides into the following commutative diagrams (and then use the definition of an $A$-module):
\begin{equation}\label{diag_hsj1}
\begin{tikzpicture}[label distance =-0.5 mm, baseline=-1 cm]	
	\node(lx1) at (5,0){};
	\node(lx2) at (2.5,0){};
	\node(ly) at (0, -2){};
	\pgfresetboundingbox;

	\node(a00) {$\srdiecko (M \tp A)$};
	\node(a01) at ($(a00)+(lx1)$) {$\srdiecko M$};
	\node(a02) at ($(a01)+(lx2)$) {$M$};
	
	\node(a10) at ($(a00)+(ly)$) {$(M \tp A) \tp A$};
	\node(a11) at ($(a01)+(ly)$) {$ M \tp A$};

	\draw[->] (a00)-- node[above] {$\srdiecko \ams$} (a01);
	\draw[->] (a01)-- node[above] {$\Xi_M$} (a02);
	
	\draw[->] (a00)-- node[left] {$s_{M\tp A}$} node[right] {$\cong$} (a10);
	\draw[->] (a01)-- node[right] {$\cong$} node[left] {$s_M$} (a11);
	\draw[<-] (a02)-- node[right, below] {$\ams$} (a11);
	
	\draw[->] (a10)-- node[below] {$\ams \tp \id_A$} (a11);
	
\end{tikzpicture}
\end{equation}

\begin{equation}\label{diag_hsj2}
\begin{tikzpicture}[label distance =-0.5 mm, baseline=-1 cm]	
	\node(lx1) at (5,0){};
	\node(lx2) at (2.5,0){};
	\node(ly) at (0, -2){};
	\pgfresetboundingbox;

	\node(a00) {$\srdiecko (M \tp A)$};
	\node(a01) at ($(a00)+(lx1)$) {$\srdiecko M$};
	\node(a02) at ($(a01)+(lx2)$) {$M$};
	
	\node(a10) at ($(a00)+(ly)$) {$(M \tp A) \tp A$};
	\node(a11) at ($(a01)+(ly)$) {$ M \tp A$};

	\draw[->] (a00)-- node[above] {$\srdiecko(\id_M \tp \eps_A)$} (a01);
	\draw[->] (a01)-- node[above] {$\Xi_M$} (a02);
	
	\draw[->] (a00)-- node[left] {$s_{M\tp A}$} node[right] {$\cong$} (a10);
	\draw[->] (a01)-- node[right] {$\cong$} node[left] {$s_M$} (a11);
	\draw[<-] (a02)-- node[right, below] {$\ams$} (a11);
	
	\draw[->] (a10)-- node[below] {$(\id_M \tp \mu_A) \circ \Phi$} (a11);
	
\end{tikzpicture}
\end{equation}
The two triangles commute by definition of $\Xi_M$. The square in (\ref{diag_hsj1}) commutes by naturality of $s$. 

The last thing to show is the commutativity of the square in (\ref{diag_hsj2}).\footnote{The naturality of $s$ does not apply here because $\id_M\tp\eps_A$ is not in $\catD$.} To show the equality of the two paths, we let $A$ act by $\sipka$ on some $T\in\catH$. The upper-right path gives\footnote{The first equality is the definition of $s_M$ and the second one is the naturality of $\Diamond_{M,T}:\srdiecko M \tp T \to T \tp M$ in $M$.}
$$ \input{\cesta1/3daco4} =  \input{\cesta1/3daco3} = \input{\cesta1/3daco2} $$
and the left-bottom path is
$$ \input{\cesta1/3daco5} = \input{\cesta1/3daco6} = \input{\cesta1/3daco1} $$
Using the observation \ref{obs_sipka_a_epsA} we see that they are equal.

\end{proof}

\subsection{Definition of $\zeta:\;  \id_\catA \to \srdiecko\circ \budzogan $} Take the composition of $\catD$-morphisms
$$  \zeta_M : M \xto{\id_M \tp 1_A } M \tp A \xto\cong \srdiecko M \xto{\srdiecko p_M} \srdiecko \budzogan M. $$
When we prove that it is the inverse of $\xi_M$, $A$-linearity will follow automatically.

\subsection{Proof that $\xi_M \circ \zeta_M = \id_M$} In the diagram
$$
\begin{tikzpicture}[label distance =-0.5 mm, baseline=-1 cm]	
	\node(lx) at (3,0){};
	\node(ly1) at (0, -1.2){};
	\node(ly2) at (0, -1.5){};
	\pgfresetboundingbox;

	\node(a00) {$M$};
	\node(a10) at ($(a00)+(lx)$) {$M \tp A$};
	\node(a20) at ($(a10)+(lx)$) {$ M$};

	\node(a11) at ($(a10)+(ly1)$) {$ \srdiecko M$};
	\node(a12) at ($(a11)+(ly2)$) {$ \srdiecko \budzogan M$};

	\draw[->] (a00)-- node[above] {$\id_M \tp 1_A $} (a10);
	\draw[->] (a10)-- node[above] {$\ams $} (a20);

	\draw[->] (a00)-- node[left] {$\zeta_M $} (a12);
	\draw[<-] (a20)-- node[right] {$\xi_M $} (a12);

	\draw[<->] (a10)-- node[right] {$\cong $} (a11);
	\draw[->] (a11)-- node[right] {$\srdiecko p_M $} (a12);

\end{tikzpicture}
$$
the two triangles are commutative --- they correspond to definitions of $\zeta_M$ and $\xi_M$. The upper path $M \to M$ is $\id_M$ so the lower path must be it as well.

\subsection{Proof that $\zeta_M \circ \xi_M = \id_{\srdiecko \budzogan M}$} We precompose both sides with the epimorphism
\begin{equation}\label{map_adf}
\begin{tikzpicture}[ baseline=0 cm]	
	\node(ly) at (2,0){};
	\pgfresetboundingbox;

	\node(a00) {$M\tp A$};
	
	\node(a01) at ($(a00)+(ly)$) {$\srdiecko M$};
	
	\node(a02) at ($(a00)+2*(ly)$) {$\srdiecko \budzogan M$};

	\draw[->] (a00)-- node[above] {\scriptsize $t_M$}  node[below]{\scriptsize $\cong$} (a01);
	\draw[->>] (a01)-- node[above] {\scriptsize $\srdiecko p_M$}(a02);
	
\end{tikzpicture}
\end{equation}
The left-hand side becomes
\begin{equation}\label{map_adf1}
M\tp A \xto{\;t_M\;} \srdiecko M \xto{\;\srdiecko p_M\;} \srdiecko \budzogan M \xto{\;\xi_M\;} M \xto{\; \zeta_M \;} \srdiecko \budzogan M
\end{equation}
and we now have to prove that (\ref{map_adf1}) equals (\ref{map_adf}). From the commutative diagram\footnote{Commutativity of the middle triangle is obvious and the other two are just definitions of $\xi_M$ and $\zeta_M$.}
$$
\begin{tikzpicture}[label distance =-0.5 mm, baseline=-1 cm]	
	\node(lx) at (3.5,0){};
	\node(ly) at (0, 2){};
	\pgfresetboundingbox;

	\node(a00) {$M\tp A$};
	\node(a20) at ($(a00)+2*(lx)$) {$M \tp A$};

	\node(a01) at ($(a00)+(ly)$) {$\srdiecko M$};
	\node(a21) at ($(a01)+2*(lx)$) {$\srdiecko M$};

	\node(a02) at ($(a00)+2*(ly)$) {$\srdiecko \budzogan M$};
	\node(a12) at ($(a02)+(lx)$) {$ M$};
	\node(a22) at ($(a02)+2*(lx)$) {$\srdiecko \budzogan M$};

	\draw[->] (a02)-- node[above] {$\xi_M$} (a12);
	\draw[->] (a12)-- node[above] {$\zeta_M$} (a22);
	
	\draw[->] (a00)-- node[below] {$\ams \tp 1_A$} (a20);

	\draw[->] (a00)-- node[left] {$t_M$}  node[right]{$\cong$} (a01);
	\draw[->>] (a01)-- node[left] {$\srdiecko p_M$}(a02);
	
	\draw[->] (a00)-- node[left] {$\mu_M$}  (a12);
	\draw[->] (a12)-- node[left] {$\id_M \tp 1_A$}(a20);
	
	\draw[->] (a20)-- node[right] {$t_M$}  node[left]{$\cong$} (a21);
	\draw[->>] (a21)-- node[right] {$\srdiecko p_M$}(a22);
	
\end{tikzpicture}
$$
we see that (\ref{map_adf1}) is equal to  
\begin{equation}\label{map_adf2}
M\tp A \xto{\ams \tp 1_A} M\tp A \xto{\;t_M\;} \srdiecko M \xto{\;\srdiecko p_M\;} \srdiecko \budzogan M 
\end{equation}
So finally, the equation to prove is (\ref{map_adf}) = (\ref{map_adf2}). As usual, we let both sides act by $\Diamond_{\budzogan M, T}$ on some $T\in \catH$. The left-hand side becomes\footnote{The first equality is from naturality of $\Diamond_{M,T}$ in $M$ and the second one is the definition of $t_M$.}
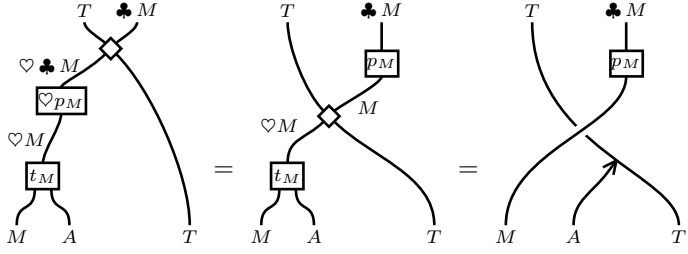
$$\input{\cesta1/2daco1}= \input{\cesta1/2daco4} = \input{\cesta1/2daco5} $$
Using this, the right-hand side (that is just the left-hand side precomposed with $\ams \tp 1_A$) becomes\footnote{The first equality holds because the right $A$-module structure on $M$ is a $\catD$-morphism. The second one is from
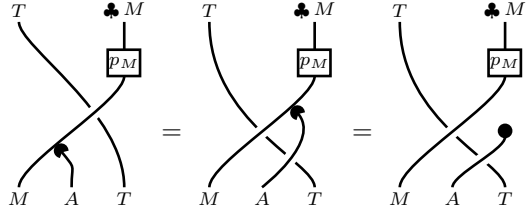
$$ \input{\cesta1/1daco3} = \input{\cesta1/1daco1}$$
which is easy to see from the definitions of $\budzogan$ and $p_M$.} 
$$ \input{\cesta1/1daco4} = \input{\cesta1/1daco5} = \input{\cesta1/1daco6} $$
The two are the same by the observation \ref{obs_sipka_a_epsA}.

\appendix


\section{$\catH=\LM H$ is a closed monoidal category}\label{section_H_is_closed}
We want to define an inner hom bifunctor 
$$\innH:\catH^{op}\times\catH\to\catH\,.$$
If $\catH$ was just the category of finite dimensional (over $\field$) modules of $H$, it would be rigid, we could take $\innH(M,N):=N\otimes\ld{M}$ and using string diagrams it would be easy to show that it works. In the general case we take instead the vector space: 
$$\innH(M,N):=Lin(M,N)$$
with the action of $H$:
\begin{equation}\label{innH_as_module}
	(h\rhd f)(\argument)=h\jedna\rhd f(Sh\dva\rhd\argument)
\end{equation}
To see $\innH$ as a bifunctor, we define\forme{\footnote{Let's remark that this is not so obvious as it seems. To see that it really corresponds in our analogy to $$(g,h)\mapsto \big(N\otimes\ld{M}\xto{h\otimes\ld{g}} N'\otimes\ld(M')\big)$$ we need to use the lemma \ref{lem_dual_of_morphisms}.}} for $g\in\catH(M',M)$ and $h\in\catH(N,N')$
$$\innH(g,h):\innH(M,N)\to\innH(M',N'): f\mapsto h\circ f\circ g\;.$$

For $\innH$ to be an internal hom we need $\innH(P,\argument)$ to be left adjoint to $\argument\otimes P$ for any fixed parameter $P\in ob(\catH)$, that is:
$$\catH(M\otimes P, N)\cong\catH(M,\innH(P,N))\,.$$
A usual approach is to find natural transformations:
\begin{itemize}
\item $\eta_M:M\to\innH(P,M\otimes P)$ corresponding to\\ $\id\in\catH(M\otimes P, M\otimes P)\cong\catH\big(M,\innH(P,M\otimes P)\big)\ni \eta$ 
\item $\eps_M:\catH(P,N)\otimes P\to N$ corresponding to\\ \quad $\id\in\catH\big(\innH(P,N),\innH(P,N)\big)\cong\catH\big(\innH(P,N)\otimes P,N\big)\in \eps$
\end{itemize}
Transcribing the string diagrams we get the formulas
\begin{equation}\label{def_eta}
	\eta(m)(\argument)=(\phi^1\rhd m)\otimes(\phi^2\cdot\beta\cdot S\phi^3\rhd\argument)
\end{equation}
\begin{equation}\label{def_eps}
	\eps(f\otimes p)=\Phi^1\rhd f(S\Phi^2\cdot\alpha\cdot\Phi^3\rhd p)\;.
\end{equation}
To be sure they really define an adjunction we need to show that
\begin{itemize}
\item $\eps$ and $\eta$ are $H$-linear;
\item the composition
\begin{equation*}
	M\otimes P\xto{\eta_M\otimes \id_P}\innH(P,M\otimes P)\otimes P\xto{\eps_{M\otimes P}}M\otimes P
\end{equation*}
is equal to $\id_{M\otimes P}$ and that
\item the composition
\begin{equation*} 
	\innH(P,M)\xto{\eta_{\innH(P,M)}}\innH(P,\innH(P,M)\otimes P)\xto{\innH(\id,\eps)}\innH(P,M)
\end{equation*}
is equal to $\id_{\innH(P,N)}$.
\end{itemize}

\subsection{Linearity of $\eta_M:M\to\innH(P,M\otimes P).$}
$$\eta(h\rhd m)(\argument)\xlongequal{(\ref{def_eta})}(\phi^1\cdot h\rhd m)\otimes(\phi^2\cdot\beta\cdot S\phi^3\rhd\argument)$$
\begin{align*}
\big(h\rhd\eta(m)\big)(\argument)& \xlongequal{(\ref{innH_as_module})} h\jedna\rhd\big(\eta(m)(h\dva\rhd\argument)\big)=\\
&\xlongequal{(\ref{def_eta})}h\jedna\rhd\big((\phi^1\rhd m)\otimes(\phi^2\cdot\beta\cdot S\phi^3\cdot h\dva\rhd\argument)\big)= \\
&=h_{(1)(1)}\cdot\phi^1\rhd m\otimes h_{(1)(2)}\cdot\phi^2\cdot\beta\cdot S\phi^3\cdot Sh\dva\rhd\argument=\\
&\xlongequal{\ref{B1}}\phi^1\cdot h\jedna\rhd m\otimes\phi^2\cdot h_{(2)(1)}\cdot\beta\cdot Sh\dvadva\cdot Sh^3\rhd\argument=\\
&\xlongequal{\ref{H2}+\ref{B3}} \phi^1\cdot h\rhd m\otimes\phi^2\cdot\beta\cdot S\phi^3\rhd\argument
\end{align*}

\subsection{Linearity of $\eps_N:\innH(P,N)\otimes P\to N$.}
\begin{align*}
\eps_N\big(h\rhd(f\otimes p)\big)&\xlongequal{(\ref{innH_as_module})} \eps_N\big(h\jednajedna\rhd f(Sh\jednadva\rhd\argument)\otimes h\dva\rhd p\big)=\\
&\xlongequal{(\ref{def_eps})}\Phi^1\cdot h\jednajedna\rhd f(Sh\jednadva\cdot S\Phi^2\cdot\alpha\cdot\Phi^3\cdot h\dva\rhd p)=\\
&\xlongequal{(\ref{B1})} h\jedna\cdot\Phi^1\rhd f(S\Phi^2\cdot Sh\dvajedna\cdot\alpha\cdot h\dvadva\cdot \Phi^3\rhd p)=\\
&\xlongequal{(\ref{H1})+(\ref{B3})}h\cdot\Phi^1\rhd f(S\Phi^2\cdot\alpha\cdot\Phi^3\rhd p)=\\
&\xlongequal{(\ref{def_eps})} h\rhd\eps_N(f\otimes p)
\end{align*}

\subsection{Composition $M\otimes P\xto{\eta_M\otimes \id_P}\innH(P,M\otimes P)\otimes P\xto{\eps_{M\otimes P}}M\otimes P$}\noindent \\
\begin{equation} \label{pomocna_formula1}
	\phii^1\otimes\phii^2_{(1)}\cdot\beta\cdot S\phii^2_{(2)}\otimes\phii^3=1\otimes\beta\otimes 1
\end{equation}
\begin{equation}\label{pomocna_formula2}
	\Phii^1\otimes\Phii^2\otimes S\Phii^3_{(1)}\cdot\alpha\cdot\Phii^3_{(2)}=1\otimes 1\otimes\alpha
\end{equation}
\begin{equation}\label{pomocna_formula3}
(\Phii^1\otimes\Phii^2\otimes\Phii^3_{(1)}\otimes\Phii^3_{(2)})\cdot 
(\Phi^1_{(1)}\otimes\Phi^1_{(2)}\otimes\Phi^2\otimes\Phi^3)\cdot(\phi^1\otimes\phi^2\otimes\phi^3\otimes 1)\cdot(\phii^1\otimes\phii^2_{(1)}\otimes\phii^2_{(2)}\otimes\phii^3)=1\otimes\Phi^1\otimes\Phi^2\otimes\Phi^3
\end{equation}
\begin{tabular}{lcl}
$m\otimes p$ &$\xmapsto{\eta_H\otimes P}$& $\big((\phi^1\rhd m)\otimes(\phi^2\cdot\beta\cdot S\phi^3\rhd\argument)\big)\otimes p$\\
& $\xmapsto{\eps_{M\otimes P}}$ & $\Phi^1\rhd\big((\phi^1\rhd m)\otimes(\phi^2\cdot\beta\cdot S\phi^3\cdot S\Phi^2\cdot\alpha\cdot\Phi^3\rhd p)\big)=$\\
\end{tabular}\noindent\\
$=(\Phi^1_{(1)}\cdot\phi^1\rhd m)\otimes(\Phi^1_{(2)}\cdot\phi^2\cdot\beta\cdot S\phi^3\cdot S\Phi^2\cdot\alpha\cdot\Phi^3\rhd p)=^ {(\ref{pomocna_formula1})}$\\
$=(\Phi^1_{(1)}\cdot\phi^1\cdot\phii^1\rhd m)\otimes(\Phi^1_{(2)}\cdot\phi^2\cdot\phii^2_{(1)}\beta\cdot S\phii^2_{(2)}\cdot S\phi^3\cdot S\Phi^2\cdot\alpha\cdot\Phi^3\cdot\phii^3\rhd p)=^ {(\ref{pomocna_formula2})}$\\
$=(\Phii^1\cdot\Phi^1_{(1)}\cdot\phi^1\cdot\phii^1\rhd m)\otimes(\Phii^2\cdot\Phi^1_{(2)}\cdot\phi^2\cdot\phii^2_{(1)}\beta\cdot S\phii^2_{(2)}\cdot S\phi^3\cdot S\Phi^2\cdot S\Phii^3_{(1)}\cdot\alpha\cdot\Phii^3_{(2)}\cdot\Phi^3\cdot\phii^3\rhd p)=^ {(\ref{pomocna_formula3})}$\\
$=m\otimes\Phi^1\cdot\beta\cdot S\Phi^2\cdot\alpha\cdot\Phi^3\rhd p=^{(\ref{H3})}$\\
$=m\otimes p$

\subsection{Composition $\innH(P,M)\xto{\eta_{\innH(P,M)}}\innH(P,\innH(P,M)\otimes P)\xto{\innH(\id,\eps)}\innH(P,M)$}\noindent\\
\begin{equation}\label{tmp1}
	\phii^1\otimes\phii^2\otimes\phii^3_{(1)}\cdot\beta\cdot\phii^3_{(2)}=1\otimes 1\otimes\beta
\end{equation}
\begin{equation}\label{tmp2}
	\Phii^1\otimes S\Phii^2_{(1)}\cdot\alpha\cdot\Phii^2_{(2)}\otimes S\Phii^3=1\otimes\alpha\otimes 1
\end{equation}
\begin{equation}\label{tmp3}
	(\Phii^1\otimes\Phii^2_{(1)}\otimes\Phii^2_{(2)}\otimes\Phii^3)\cdot(\Phi^1\otimes\Phi^2\otimes\Phi^3\otimes 1)\cdot(\phi^1\jedna\otimes\phi^1\dva\otimes\phi^2\otimes\phi^3)\cdot(\phii^1\otimes\phii^2\otimes\phii^3\jedna\otimes\phii^3\dva)=1\otimes\phi^1\otimes\phi^2\otimes\phi^3
\end{equation}
\begin{tabular}{lcl}
$f$&$\xmapsto{\eta_{\innH(P,M)}}$&$(\phi^1\rhd f)\otimes\phi^2\cdot\beta\cdot S\phi^3\rhd\argument=$\\
&&$=\big(\phi^1_{(1)}\rhd f(S\phi^1_{(2)}\rhd\widehat\argument)\big)\otimes\big(\phi^2\cdot\beta\cdot S\phi^3\rhd\argument\big)$\\
&$\xmapsto{\innH(id_P,\eps)}$&$\Phi^1\cdot\phi^1_{(1)}\rhd f(S\phi^1_{(2)}\cdot S\Phi^2\cdot\alpha\cdot\Phi^3\cdot\phi^2\cdot\beta\cdot S\phi^3\rhd\argument)=^{(\ref{tmp1})}$
\end{tabular}\noindent\\
$=\Phi^1\cdot\phi^1_{(1)}\cdot\phii^1\rhd f(S\phii^2\cdot S\phi^1_{(2)}\cdot S\Phi^2\cdot\alpha\cdot\Phi^3\cdot\phi^2\cdot\phii^3_{(1)}\cdot\beta\cdot S\phii^3_{(2)}\cdot S\phi^3\rhd\argument)=^{(\ref{tmp2})}$
$=\Phii^1\cdot\Phi^1\cdot\phi^1_{(1)}\cdot\phii^1\rhd f(S\phii^2\cdot S\phi^1_{(2)}\cdot S\Phi^2\cdot S\Phii^2_{(1)}\cdot\alpha\cdot\Phii^2_{(2)}\cdot\Phi^3\cdot\phi^2\cdot\phii^3_{(1)}\cdot\beta\cdot S\phii^3_{(2)}\cdot S\phi^3\cdot S\Phii^3\rhd\argument)=^{(\ref{tmp3})}$
$=f(S\phi^1\cdot\alpha\cdot\phi^2\cdot\beta\cdot S\phi^3\rhd\argument)=^{(\ref{H4})}$
$=f(\argument)$


\section{Explicit Madness}
In this section, $\catH$ will be the category of representations of a quasi-Hopf algebra $H$.  We do the proofs first in the rigid subcategory $\catHfd$ of finite dimensional $H$-modules because they are much simpler and they motivate the formulas. 

\begin{lem}
The natural map $M\tp \innH( X, Y ) \to \innH( X, M\tp Y )$ is given by
\begin{equation}\label{equa_dnu}
 m\tp f \;\; \mapsto \;\; (\phi^1 \rhd m ) \tp \Big( \phi^2 \rhd f(S\phi^3 \rhd \argument ) \Big). 
\end{equation}
\end{lem}
\begin{proof}
In the rigid $\catHfd$, our map is just the associativity morphism $$M\tp (Y \tp\ld X) \xto\phi (M \tp Y) \tp \ld X$$ which gives precisely the formula above.

Consider the general case and denote our map by $\rho$. Then it is uniquely determined by the commutativity of
\begin{equation}\label{diag_asdfj}
\begin{tikzpicture}[label distance =-0.5 mm, baseline=-1 cm]	
	\node(lx1) at (5.5,0){};
	\node(lx2) at (2.5,0){};
	\node(lx3) at (1.5,0){};
	\node(lx4) at (3.1,0){};
	\node(ly) at (0, -2){};
	\pgfresetboundingbox;

	\node(a00) {$ \Big( M \tp \innH ( X,Y) \Big) \tp X$};
	\node(a01) at ($(a00)+(lx1)$) {$ M \tp \Big( \innH ( X, Y ) \tp X \Big)$};
	
	\node(a10) at ($(a00)+(ly)$) {$ \innH \big( X , M\tp Y \big) \tp X$};
	\node(a11) at ($(a01)+(ly)$) {$ M \tp Y$};

	\draw[->] (a00)-- node[above] {$\Phi $} (a01);

	\draw[->] (a00)-- node[left] {$\rho \tp \id_X $} (a10);
	\draw[->] (a01)-- node[right] {$\id_M \tp \eeps_{Y,X} $} (a11);
	
	\draw[->] (a10)-- node[below] {$ \eeps_{ M \tp Y, X} $} (a11);

\end{tikzpicture}
\end{equation}
So we just need to plug $\rho$ defined by (\ref{equa_dnu}) into (\ref{diag_asdfj}) and verify that it commutes. The upper-right path is 
\begin{align*}
& m \tp f \tp x \xmapsto{ \Phi}  \Big(\Phi^1 \rhd m \Big) \tp \Big( \Phi^2_{(1)} \rhd f\big( S \Phi^2_{(2)} \rhd \argument \big) \Big) \tp \Big( \Phi^3 \rhd x \Big) \mapsto \\
& \xmapsto{\id_M \tp \eeps_{Y,X} } \Big( \Phi^1 \rhd m \Big) \tp \bigg( \Phii^1 \cdot \Phi^2_{(1)} \rhd f \Big[ S( \Phii^2\cdot \Phi^2_{(2)} ) \cdot   \alpha \cdot \Phii^3 \cdot \Phi^3 \rhd x \Big] \bigg) 
\end{align*}
and  the left-lower path gives
\begin{align*}
& m \tp f \tp x \xmapsto{ \rho \tp \id_X} \rho(m\tp f ) \tp x    \xmapsto{ \eeps_{M \tp Y, X} }   \Phi^1 \rhd \Big[ \rho( m \tp f )\big( S\Phi^2 \cdot \alpha \cdot \Phi^3 \rhd x \big) \Big] = \\
& \quad\quad \xlongequal{1}\Big( \Phi^1_{(1)} \cdot \phi^1 \rhd m \Big) \tp \bigg( \Phi^1_{(2)} \cdot \phi^2 \rhd f\Big( S( \Phi^2 \cdot \phi^3 ) \cdot \alpha \cdot \Phi^3 \rhd x \Big) \bigg)= \\ 
&  \xlongequal{2}\Big( \Phii^1 \cdot \Phi^1_{(1)} \cdot \phi^1 \rhd m \Big) \tp \bigg( \Phii^2\cdot \Phi^1_{(2)} \cdot \phi^2 \rhd f\Big( S( \Phii^3_{(1)} \cdot \Phi^2 \cdot \phi^3 ) \cdot \alpha \cdot \Phii^3_{(2)} \cdot \Phi^3 \rhd x \Big) \bigg) \\
\end{align*}
where in the first equality we replaced $\rho$ by (\ref{equa_dnu}) and in the second equality we used (\ref{pomocna_formula2}). To show that the results produced by the two paths are the same one just uses the pentagon in the form
\begin{align*}
& \Phi^1  \tp \big( \Phii^1 \cdot \Phi^2_{(1)} \big) \tp \big( \Phii^2\cdot \Phi^2_{(2)} \big) \tp \big( \Phii^3 \cdot \Phi^3  \big) = \\
& \quad\quad = \big( \Phii^1 \cdot \Phi^1_{(1)} \cdot \phi^1  \big) \tp \big( \Phii^2\cdot \Phi^1_{(2)} \cdot \phi^2 \big) \tp \big( \Phii^3_{(1)} \cdot \Phi^2 \cdot \phi^3  \big) \tp \big( \Phii^3_{(2)} \cdot \Phi^3  \big).
\end{align*}
\end{proof}

\begin{cor}\label{cor_skjfss}
If we identify $\int_{X \in \catH}   \innH\big( X , N\tp (X\tp M)\big) $ with $N\tp H \tp M$ and $\srdiecko M = \int_{X \in \catH} \innH( X, X\tp M) $ with $ H\tp M$ as in Lemma \ref{lem_explicit_srdiecko} then the natural map
$$ N \tp \srdiecko M \to  \innH\big( X , N\tp (X\tp M)\big) $$
is $ n\tp a \tp m \mapsto (\phi^1 \rhd n ) \tp ( \phi^2_{(1)} \cdot a \cdot S \phi^3) \tp (\phi^2_{(2)} \rhd m )$.

In particular, it is an isomorphism.
\end{cor}
\begin{proof}
An element $ n\tp l_{a,m} \in N \tp \srdiecko M \subset N \tp \innH(C, C\tp M) $ is sent to an element of $\innH\big(C, N\tp (C\tp M) \big) $ (i.e.  map $ C \to N\tp C \tp M $) given by 
\begin{align*}
c\mapsto\quad & \big(\phi^1 \rhd n \big) \tp \Big( \phi^2 \rhd l_{a,m} ( S\phi^3 \rhd c ) \Big) = \\
= & \big( \phi^1 \rhd n \big) \tp \Big( \phi^2 \rhd \big[ ( a \cdot S\phi^3 \cdot c ) \tp  m \big] \Big) = \\
= & \big( \phi^1 \rhd n \big) \tp  \big( \phi^2_{(1)} \cdot a \cdot S\phi^3 \ \cdot c \big) \tp \big( \phi^2_{(2)} \rhd m \big) = \\
=& l_{t}  (c). 
\end{align*}
where 
$ t = \big( \phi^1 \rhd n \big) \tp  \big( \phi^2_{(1)} \cdot a \cdot S\phi^3  \big) \tp \big( \phi^2_{(2)} \rhd m \big) \;\in \;  N \tp H \tp M $.
\end{proof}
%

\begin{lem}
The ``inner composition'' $\innH(Y,Z) \otimes \innH(X, Y) \xto{\comp} \innH( X, Z)$ is 
\begin{align} \label{equa_comp} 
g \otimes f \mapsto \Phi^1 \rhd g\bigg[ S( \phi^1\cdot \Phi^2 ) \cdot \alpha \cdot \phi^2 \cdot \Phi^3_{(1)} \rhd f \Big( S( \phi^3 \cdot \Phi^3_{(2)} ) \rhd \argument \Big) \bigg]
\end{align}
\end{lem}
\begin{proof}
In the rigid category $\catHfd$ our $\comp$ is just the composition 
\begin{align*}
(Z \tp \ld Y ) \tp ( Y \tp \ld X) &\xto{ \Phi_{1,2,(3,4)} } Z \tp \Big( \ld Y \tp ( Y\tp \ld X ) \Big)  \to \\
& \xto{ \id \tp \phi } Z \tp \Big( ( \ld Y \tp Y ) \tp \ld X \Big) \xto{ \id \tp ( ev \tp \id )} Z \tp \ld X
\end{align*}
\begin{align*}
g\tp f &\xmapsto{ \Phi_{1,2,(3,4)} } \bigg[ \Phi^1 \rhd g \big( S \Phi^2 \rhd \argument \big) \bigg] \otimes \bigg[\Phi^3_{(1)} \rhd f \big( S\Phi^3_{(2)} \rhd \argument \big) \bigg] \mapsto \\
&\xmapsto{\id \tp\phi } \bigg[ \Phi^1 \rhd g \Big( S(\phi^1 \cdot \Phi^2) \rhd \argument \Big) \bigg] \otimes \bigg[\phi^2 \cdot \Phi^3_{(1)} \rhd f \Big( S( \phi^3 \cdot \Phi^3_{(2)})  \rhd \argument \Big) \bigg] \mapsto \\
& \xmapsto{ \id \tp ( ev \tp \id )}  \Phi^1 \rhd g\bigg[ S( \phi^1\cdot \Phi^2 ) \cdot \alpha \cdot \phi^2 \cdot \Phi^3_{(1)} \rhd f \Big( S( \phi^3 \cdot \Phi^3_{(2)} ) \rhd \argument \Big) \bigg]
\end{align*}

For a general closed monoidal category, $\comp$ is uniquely determined by commutativity of
$$ 
\begin{tikzpicture}[baseline=-3cm]	
	\node(lx) at (2,0){};
	\node(ly) at (0, 2){};
	\pgfresetboundingbox;

	\node(a01) {$\Big( \innH( Y, Z ) \tp \innH( X, Y ) \Big) \tp X $};

	\node(a10) at ($(a01)+(ly)+0.5*(lx)$) {$ \innH( Y, Z ) \tp \Big( \innH( X, Y )  \tp X \Big) $};

	\node(a20) at ($(a01)+(ly)+3.5*(lx)$) {$ \innH( Y, Z )  \tp Y $};

	\node(a31) at ($(a01)+4*(lx)$) {$ Z $};

	\node(a12) at ($(a01)-(ly)+2*(lx)$) {$ \innH( X, Z )  \tp X $};

	\draw[->] (a01)-- node[left] {$\Phi$} (a10);
	\draw[->] (a10)-- node[above] {$\id \tp \eeps_{Y,X}$} (a20);
	\draw[->] (a20)-- node[right] {$\eeps_{Z,Y}$} (a31);
	\draw[->] (a01)-- node[left] {$\comp \tp\id_X$} (a12);
	\draw[->] (a12)-- node[right] {$\quad \eeps_{Z,X}$} (a31);
\end{tikzpicture}
$$
So to prove that $\comp$ is given by the formula (\ref{equa_comp}) we need just to plug it into the diagram.
The upper path gives
\begin{align*}
& ( g \tp f )\tp x \xmapsto{\Phi} \Big( \Phi^1_{(1)} \rhd g( S\Phi^1_{(2)}\rhd \argument ) \Big)\tp \Big( \Phi^2_{(1)} \rhd f( S \Phi^2_{(2)} \rhd \argument ) \Big) \tp \Big( \Phi^3 \rhd x \Big) \\
& \xmapsto{ \id \tp \eeps_{Y,X} } \Big( \Phi^1_{(1)} \rhd g( S \Phi^1_{(2)} \rhd \argument ) \Big) \tp \bigg( S \Phii^1 \cdot \Phi^2_{(1)} \rhd f \Big( S( \Phii^2 \cdot \Phi^2_{(2)} ) \cdot \alpha \cdot \Phii^3 \cdot \Phi^3 \rhd x \Big) \bigg) \\
& \xmapsto{ \eeps_{Z,Y} } \Phiii^1 \cdot \Phi^1_{(1)} \rhd g \bigg( S ( \Phiii^2 \cdot \Phi^1_{(2)} ) \cdot \alpha \cdot \Phiii^3 \cdot \Phii^1 \cdot \Phi^2_{(1)} \rhd f \Big( S ( \Phii^2 \cdot \Phi^2_{(2)} ) \cdot \alpha \cdot \Phii^3 \cdot \Phi^3 \rhd x \Big) \bigg) \\
& = \Phiii^1 \cdot \phi^1_{(1)} \cdot \Phi^1_{(1)} \rhd g \bigg( S ( \Phiii^2 \cdot \phi^1_{(2)} \cdot \Phi^1_{(2)} ) \cdot \alpha \cdot \Phiii^3 \cdot \phi^2 \cdot \Phii^1 \cdot \Phi^2_{(1)} \rhd \\
& \quad\quad \quad\quad\quad\quad\quad\quad \quad\quad\quad\quad\rhd f \Big( S ( \phi^3_{(1)}\cdot \Phii^2 \cdot \Phi^2_{(2)} ) \cdot \alpha \cdot \phi^3_{(2)} \cdot \Phii^3 \cdot \Phi^3 \rhd x \Big) \bigg), \\
\end{align*}
where in the last equality we used
$$ 1 \tp 1 \tp 1 \tp \alpha = \phi^1_{(1)} \tp \phi^1_{(2)} \tp \phi^2 \tp\big( S\phi^3_{(1)} \cdot \alpha \cdot \phi^3_{(2)} \big). $$
The lower path is
\begin{align*}
& (g\tp f) \tp x \xmapsto{ \comp \tp \id_X } \bigg( \Phi^1 \rhd g \Big[ S ( \phi^1 \cdot \Phi^2 ) \cdot \alpha \cdot \phi^2 \cdot \Phi^3_{(1)} \rhd f \Big( S( \phi^3 \cdot \Phi^3_{(2)} ) \rhd \argument \Big) \Big] \bigg) \tp x \\ 
& \xmapsto{\eeps_{Z,X} } \Phii^1 \cdot \Phi^1 \rhd g \bigg[ S(\phi^1 \cdot \Phi^2)\cdot \alpha \cdot \phi^2 \cdot \Phi^3_{(1)} \rhd f \Big( S( \Phii^2 \cdot \phi^3 \cdot \Phi^3_{(2)} ) \cdot \alpha \cdot \Phii^3 \rhd x \Big) \bigg] = \\
& \xlongequal{(\ref{equa_hsko})} \Phii^1 \cdot \phii^1 \cdot \Phi^1 \rhd g \bigg[ S(\phii^2_{(1)} \cdot \phi^1 \cdot \Phi^2)\cdot \alpha \cdot \phii^2_{(2)} \cdot \phi^2 \cdot \Phi^3_{(1)} \rhd \\ 
& \quad\quad \quad\quad\quad\quad\quad\quad \quad\quad\quad\quad\rhd f \Big( S( \Phii^2 \cdot \phii^3 \cdot \phi^3 \cdot \Phi^3_{(2)} ) \cdot \alpha \cdot \Phii^3 \rhd x \Big) \bigg], \\
& \xlongequal{(\ref{equa_hskoo})} \Phii^1_{(1)} \cdot \phii^1 \cdot \Phi^1 \rhd g \bigg[ S(\Phii^1_{(2)(1)} \cdot \phii^2_{(1)} \cdot \phi^1 \cdot \Phi^2)\cdot \alpha \cdot \Phii^1_{(2)(2)}\cdot \phii^2_{(2)} \cdot \phi^2 \cdot \Phi^3_{(1)} \rhd \\ 
& \quad\quad \quad\quad\quad\quad\quad\quad \quad\quad\quad\quad\rhd f \Big( S( \Phii^2 \cdot \phii^3 \cdot \phi^3 \cdot \Phi^3_{(2)} ) \cdot \alpha \cdot \Phii^3 \rhd x \Big) \bigg], \\
\end{align*}
where the two equalities come from
\begin{equation}\label{equa_hsko}
 1 \tp \alpha \tp 1 =  \phii^1 \tp \big( S\phii^2_{(1)} \cdot \alpha \cdot \phii^2_{(2)} \big) \tp \phii^3 
\end{equation}
\begin{equation}\label{equa_hskoo}
\Phii^1 \tp \alpha \tp \Phii^2 \tp \Phii^3  =  \Phii^1_{(1)} \tp \Big( S\Phii^1_{(2)(1)} \cdot \alpha \cdot \Phii^1_{(2)(2)} \Big) \tp \Phii^2 \tp \Phii^3 .
\end{equation}
Now it is enough to see that in $H^{\tp 5}$
\begin{align*}
& (\Phiii^1 \cdot \phi^1_{(1)} \cdot \Phi^1_{(1)} ) \tp  ( \Phiii^2 \cdot \phi^1_{(2)} \cdot \Phi^1_{(2)} ) \tp  ( \Phiii^3 \cdot \phi^2 \cdot \Phii^1 \cdot \Phi^2_{(1)}) \tp \\
& \quad\quad \quad\quad\quad\quad\quad\quad \quad\quad\quad\quad \tp ( \phi^3_{(1)}\cdot \Phii^2 \cdot \Phi^2_{(2)} ) \tp ( \phi^3_{(2)} \cdot \Phii^3 \cdot \Phi^3  ) =  \\
& =( \Phii^1_{(1)} \cdot \phii^1 \cdot \Phi^1 ) \tp (\Phii^1_{(2)(1)} \cdot \phii^2_{(1)} \cdot \phi^1 \cdot \Phi^2) \tp ( \Phii^1_{(2)(2)}\cdot \phii^2_{(2)} \cdot \phi^2 \cdot \Phi^3_{(1)}) \tp \\ 
& \quad\quad \quad\quad\quad\quad\quad\quad \quad\quad\quad\quad\tp ( \Phii^2 \cdot \phii^3 \cdot \phi^3 \cdot \Phi^3_{(2)} ) \tp ( \Phii^3 ) \\
\\
\end{align*}
The left-hand side corresponds to the rebracketing
\begin{align*}
\big( (\bodka \bodka)  (\bodka \bodka) \big) \bodka \xto{\Phi} (\bodka \bodka ) \big((\bodka \bodka ) \bodka \big) \xto{\Phii} (\bodka \bodka ) \big( \bodka (\bodka \bodka ) \big) \xto{\phi} \big( ( \bodka \bodka )\bodka \big) ( \bodka \bodka ) \xto{\Phiii} \big( \bodka (\bodka \bodka ) \big) ( \bodka\bodka )
\end{align*}
and the right-hand side to
\begin{align*}
\big( (\bodka \bodka)  (\bodka \bodka) \big) \bodka \xto{\Phi} \Big(\bodka \big( \bodka(\bodka\bodka) \big) \Big) \bodka \xto{\phi} \Big( \bodka \big( (\bodka\bodka) \bodka\big) \Big) \bodka \xto{\phii} \Big( \big( \bodka( \bodka\bodka) \big) \bodka\Big) \bodka \xto{\Phii}  \big( \bodka (\bodka \bodka ) \big) ( \bodka\bodka ) 
\end{align*}
so they are equal by the coherence theorem of monoidal categories.
\end{proof}

\begin{cor}\label{cor_product_formulas}
If we identify $A$ with $H$ and $\srdiecko M$ with $H\tp M$ via linear maps
$$ H \xto\cong A: a \mapsto l_a;\quad \quad H\tp M \xto\cong \srdiecko M: a\tp m \mapsto l_{a\tp m} $$
then the formulas for the product in $A$ and the right $A$-module structure on $\srdiecko M$ are
$$a \comp b =  \Phi^1 \cdot a \cdot S ( \phi^1 \cdot \Phi^2 )\cdot \alpha \cdot \phi^2 \cdot \Phi^3_{(1)} \cdot b \cdot S( \phi^3 \cdot \Phi^3_{(2)} )  $$ 
 $$(a\tp m) \comp b = \bigg( \Phi^1_{(1)} \cdot a \cdot S ( \phi^1 \cdot \Phi^2 )\cdot \alpha \cdot \phi^2 \cdot \Phi^3_{(1)} \cdot b \cdot S( \phi^3 \cdot \Phi^3_{(2)} ) \bigg) \tp \Big( \Phi^1_{(2)} \rhd m \Big) $$  
\end{cor}

\begin{claim}\label{claim_eps_on_Phi}
If we apply $\eps$ to any component of $\Phi$ or $\phi$, we get $1$. For instance
$$ \Phi^1 \tp \Phi^2 \tp \eps \Phi^3 = 1_H \tp 1_H. $$
\end{claim}
\begin{proof}
The above formula corresponds to the diagram
$$
\begin{tikzpicture}
	\node(vzdx) at (4, 0){};
	\node(vzdy) at (0, -1.5){};
	
	\node(a) {$(M \tp N ) \tp \field $};
	\node(b) at ($(a)+(vzdx)$) {$M \tp ( N \tp \field )$};
	\node(c) at ($(b)+(vzdy)$) {$M\tp N$};
	
	\draw[->] (a)-- node[above] {$\Phi$} (b);
	\draw[->] (a)-- node[below] {$\cong$} (c);
	\draw[->] (b)-- node[right] {$\cong$} (c);	
\end{tikzpicture}
$$
that is commutative since $\catH$ is a monoidal category. 
\end{proof}

\begin{lem}\label{lem_kopec_formul}
For any $X\in \catH $, the action $\sipka_X: A\tp X \to X$ is given by
\begin{equation}
 a \sipka x \, = \, \Phi^1 \cdot a \cdot S\Phi^2 \cdot \alpha \cdot \Phi^3 \rhd x . 
\end{equation}
In particular, the augmentation $\eps_A : A\to \field$ is 
\begin{equation}
 \eps_A(a) = \eps(a) \cdot \eps(\alpha).
\end{equation}

More generally, if $M \in \catH$, the formula for $\Diamond_{M,X} :\srdiecko M \tp X \to X \tp M$ is 
\begin{equation}\label{explicit_Diamond}
\Diamond_{M,X}\Big( (a\tp m ) \tp x \Big) = \Big( \Phi^1_{(1)} \cdot a \cdot S\Phi^2 \cdot \alpha \cdot \Phi^3 \rhd x \Big) \tp \Big( \Phi^1_{(2)} \rhd m \Big) 
\end{equation}
and the natural projection $\pi_M: \srdiecko M \to M$ is
\begin{equation}\label{explicit_pi}
\pi_M( a\tp m) \, = \, \eps( a\cdot \alpha) \cdot m.
\end{equation}
\end{lem}
\begin{proof}
Let's prove (\ref{explicit_Diamond}). The canonical map
$$ \srdiecko M = \int_{X\in \catH } \innH( X, X\tp M)  \;\longrightarrow\; \innH(X, X\tp M) $$
maps $a\tp m \in \srdiecko M$ to the ``left multiplication'' 
$$ l^X_{a\tp m} :\; X\to X\tp M :\; x \mapsto (a \rhd x) \tp m .$$
Now one just uses this and the formula (\ref{def_eeps}) to compute $\Diamond_{M,X}$ as the composition
$$ \Diamond_{M,X} : \srdiecko M \tp X \to \innH(X,X\tp M) \tp X \xto{ \eeps_{X\tp M, X} } X \tp M . $$ 
To show (\ref{explicit_pi}) just realize that $\pi_M(a\tp m) = \Diamond_{M,\field}\big( (a\tp m) \tp 1_\field \big)$
and apply Claim \ref{claim_eps_on_Phi} to get rid of $\Phi$.
\end{proof}

\bibliographystyle{amsplain}

\begin{thebibliography}{99}
\bibitem{Bulacu1} D. Bulacu, F. Panaite, F.V. Oystaeyen, Quasi-Hopf algebra actions and smash product, Comm. Algebra 28 (2) (2000), 631--651.  
\bibitem{Bulacu2} D. Bulacu, E. Nauwalearts: Radford’s biproduct for quasi-Hopf algebras and bosonization,
J. Pure Appl. Algebra 174 (2002), 1--42. 
\bibitem{Bulacu3} Bulacu, Caenepeel and Panaite: More properties of Yetter-Drinfeld modules over quasi-Hopf algebras, Lecture Notes Pure Appl. Math. 239 (2004). 
\bibitem{HN} F. Hausser, F. Nill: Integral Theory for Quasi-Hopf Algebras; arXiv:math/9904164 [math.QA].
\bibitem{EK2}  P. Etingof, D. Kazhdan: Quantization of Lie bialgebras. II, III. Selecta Math. (N.S.) 4 (1998), no. 2, 213--231, 233--269. 
\bibitem{Drinfeld1} V. G. Drinfeld: Quasi-Hopf algebras. (Russian) Algebra i Analiz 1 (1989), no. 6, 114--148; translation in Leningrad Math. J. 1 (1990), no. 6, 1419--1457.
\bibitem{SS} \v S. Sak\'alo\v s, P. \v Severa: On quantization of quasi-Lie bialgebras; arXiv:1304.6382 [math.QA].
\end{thebibliography}

\end{document}